\newif\ifprint
\printfalse
\documentclass[twoside,10pt]{HandbookOfModuli}

\usepackage{amsmath,amsthm}
\usepackage{amssymb}
\usepackage{latexsym}
\usepackage[utf8]{inputenc}
\usepackage{textcomp}
\usepackage{color}
\usepackage[pdftex]{graphicx}
\usepackage{titlesec}
\usepackage{xspace}

\ifprint
	\input{giovanni} 
	\usepackage[final]{microtype}
\fi
\usepackage{bm}
\renewcommand{\mathbf}[1]{\bm{#1}} 

\ifprint
	\definecolor{linkred}{rgb}{0,0,0} 
	\definecolor{linkblue}{rgb}{0,0,0} 
\else
	\definecolor{linkred}{rgb}{0.7,0.2,0.2}
	\definecolor{linkblue}{rgb}{0,0.2,0.6}
\fi

\numberwithin{equation}{section} 

\usepackage[
	hypertexnames=false,
	hyperindex,
	pagebackref,
	pdftex,
	pdftitle={Handbook of Moduli},
	pdfdisplaydoctitle,
	pdfpagemode=UseNone,
	breaklinks=true,
	extension=pdf,
	bookmarks=false,
	plainpages=false,
	colorlinks,
	linkcolor=linkblue,
	citecolor=linkred,
	urlcolor=linkred,
	pdfmenubar=true,
	pdftoolbar=true,
	pdfpagelabels,
	pdfpagelayout=TwoPage,
	pdfview=Fit,
	pdfstartview=Fit
]{hyperref}

\catcode`@=11 \baselineskip 4.5mm
\def\ps@handbook{\def\@oddhead{\hfill \leftmark \hfill\thepage }
\def\@evenhead{\thepage \hfill \rightmark \hfill}
\def\@oddfoot{}
\def\@evenfoot{}}
\def\@evenhead{}
\def\@oddfoot{}
\def\@evenfoot{\hfill\copyright\ China Higher Education Press}
\def\list#1#2{\ifnum \@listdepth >5\relax \@toodeep \else \global
\advance \@listdepth\@ne \fi \rightmargin \z@ \listparindent\z@
\itemindent\z@ \csname @list\romannumeral\the\@listdepth\endcsname
\def\@itemlabel{#1}\let\makelabel\@mklab \@nmbrlistfalse #2\relax
\@trivlist \parskip -\parsep \parindent\listparindent \advance
\linewidth -\rightmargin \advance\linewidth -\leftmargin \advance
\@totalleftmargin \leftmargin \parshape \@ne \@totalleftmargin
\linewidth \ignorespaces}
\renewcommand*\l@section{\@tocline{1}{0pt}{0em}{1em}{}}
\renewcommand*\l@subsection{\@tocline{2}{0pt}{1.5em}{2em}{}} 
\renewcommand\contentsnamefont{\bfseries\large} 
\catcode`@=12
\renewcommand{\theequation}{\thesection.\arabic{equation}}

\def\thebibliography#1{\section*{References}
\list{[\arabic{enumi}]}{\settowidth \labelwidth{[#1]} \leftmargin
\labelwidth \advance \leftmargin \labelsep \usecounter{enumi}}
\def\newblock{\hskip .11em plus .33em minus .07em} \sloppy
\clubpenalty 4000 \widowpenalty 4000 \sfcode`\.=1000 \relax}

\titleformat{\section}{\normalfont\large\bfseries}{\thesection.}{0.5em}{}[\kern0.em]
\titleformat{\subsection}{\normalfont\bfseries}{\thesubsection.}{0.3em}{}[\kern0.em]
\titleformat{\subsubsection}[runin]{\normalfont\bfseries}{\thesubsubsection.}{0.5em}{}[\kern0.5em]

\def\fofsubsubsection#1{\refstepcounter{equation}\subsubsection*{\theequation.\kern0.25em #1}}
\def\foisubsubsection#1{\refstepcounter{equation}\subsubsection*{\kern\parindent\theequation.\kern0.25em #1}}

\usepackage{amssymb,amscd}
\usepackage{graphicx}

\usepackage{latexsym}
\usepackage{multirow}
\usepackage{setspace}

\usepackage{amsthm}

\newtheorem{defn}[equation]{Definition}
\newtheorem{ex}{Example}
\newtheorem{lemma}[equation]{Lemma}
\newtheorem{cor}[equation]{Corollary}
\newtheorem{prop}[equation]{Proposition}
\newtheorem{thm}[equation]{Theorem}

\newcommand{\res}[1]{\begin{array}[d]{l}\\{\rm Res}\\^{#1}\end{array}\hspace{-1mm}}

\newcommand{\bc}{\mathbb{C}}

\newcommand{\bp}{\mathbb{P}}
\newcommand{\bq}{\mathbb{Q}}
\newcommand{\br}{\mathbb{R}}
\newcommand{\bz}{\mathbb{Z}}
\newcommand{\modm}{\mathcal{M}}

\newcommand{\bb}{{\mathbf b}}

\newcommand{\f}{\mathcal{F}}

\newcommand{\lc}{\mathcal{L}}

\newcommand{\tr}{{\rm tr}\hspace{.2mm}}

\newcommand{\fat}{\f\hspace{-.3mm}{\rm at}_{g,n}}
\newcommand{\fato}{\f\hspace{-.3mm}{\rm at}}

\begin{document}

\title[Moduli space, lattice points and Hurwitz problems]{Cell decompositions of moduli space, lattice points and Hurwitz problems}
\author{Paul Norbury}
\address{\hspace{-5mm}Department of Mathematics and Statistics\\
University of Melbourne\\Australia 3010}
\email{pnorbury@ms.unimelb.edu.au}
\keywords{}
\subjclass[2000]{32G15; 30F30; 05A15}
\date{}

\begin{abstract}

\noindent In this article we describe the cell decompositions of the moduli space of Riemann surfaces due to Harer, Mumford and Penner and its relationship to a Hurwitz problem.  The cells possess natural linear structures and with respect to this they can be described as rational convex polytopes.  We show how to effectively calculate the number of lattice points and the volumes over all cells.  These calculations contain deep information about the moduli space.

\end{abstract}

\maketitle

\tableofcontents

\section{Introduction}

Let $\modm_{g,n}$ be the moduli space of genus $g$ Riemann surfaces with $n$ labeled points for $2-2g-n<0$.  A Riemann surface gives a conformal class of metrics on an oriented surface, and in particular negative Euler characteristic guarantees in each conformal class a unique complete hyperbolic metric on the complement of the $n$ labeled points.  Thus we can equivalently define $\modm_{g,n}$ to be the moduli space of oriented genus $g$ hyperbolic surfaces with $n$ labeled cusps.  A third characterisation of the moduli space uses the analogous fact that each conformal class contains a singular {\em flat} metric with cone angles.  More precisely, the complement of the set of labeled points in a Riemann surface uniquely decomposes into a union of $n$ flat cylinders of unit circumference  \cite{StrQua} which meet along a 1-dimensional graph on the surface known as a fatgraph.  Each edge of the graph inherits a length and a vertex of the graph of valency $\nu$ corresponds to a cone angle on the surface of angle $\nu\pi$.  The fatgraphs also arise from the hyperbolic perspective essentially as cut loci of unit horocycles around cusps \cite{PenDec}.

The third description of the moduli space induces a natural cell decomposition of the moduli space where each cell is indexed by a fatgraph.  The cells have a number of good properties.  They have a natural linear structure and with respect to this they are convex polytopes.  They have a natural volume form which was used by Kontsevich \cite{KonInt}  to define a volume $V_{g,n}$ of the moduli space.  There are natural integer points inside the cells and they are naturally rational convex polytopes.  

Many different cell decompositions of the moduli space arise if one replaces the unit circumference condition, in the decomposition of a Riemann surface  into a union of $n$ flat cylinders, with $n$ fixed circumferences $(b_1,...,b_n)$.  As before, the cells are convex polytopes with respect to a linear structure and equipped with a natural volume form.  Kontsevich's volume depends on the cell decomposition and becomes a function $V_{g,n}(b_1,...,b_n)$ which he proved is polynomial in the $b_i$.  When the $b_i$ are integers there are natural integer points inside each cell.  These integer points correspond to a finite set of algebraic curves defined over the algebraic numbers that  are represented as covers of $S^2$ branched over three points.  Their fatgraphs - which define the intersection of the flat cylinders - have integer edge lengths.

For example, the moduli space $\modm_{0,4}$ has a very simple cell decomposition given by two triangles glued along their perimeters.  The diagram shows one of these triangles with its integer points which correspond to algebraic numbers in $\overline{\bc}-\{0,1,\infty\}$.  These algebraic numbers are calculated in Example~\ref{ex:dessin}.
\begin{figure}[ht]  
	\centerline{\includegraphics[height=2.5cm]{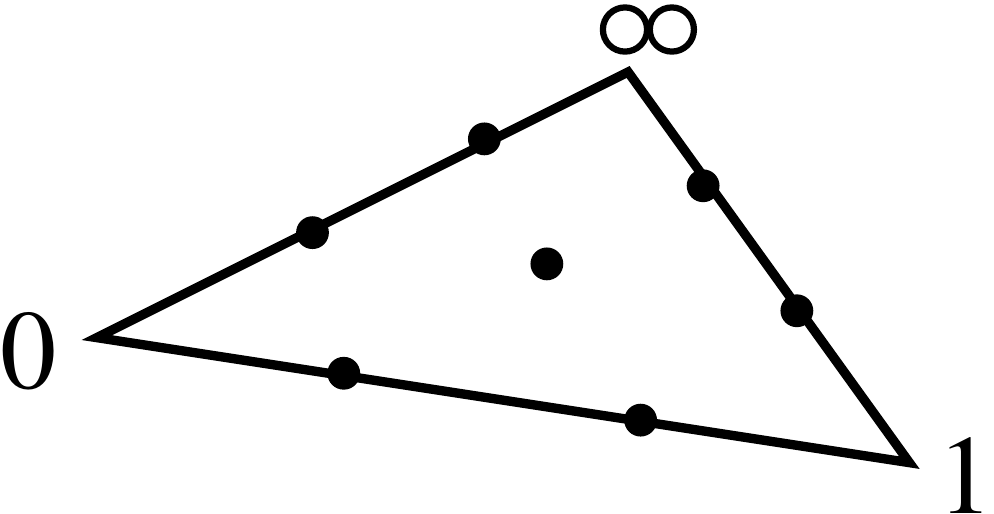}}
	\caption{Cell decomposition of $\modm_{0,4}$}
	\label{fig:cell}
\end{figure}

It is natural to count the Riemann surfaces corresponding to integer points inside cells.  The number of lattice points depends on the cell decomposition and becomes a function $N_{g,n}(b_1,...,b_n)$ for $(b_1,...,b_n)\in\bz^n$  \cite{NorCou} analogous to the volume $V_{g,n}(b_1,...,b_n)$.   It is quasi-polynomial in the $b_i$ (meaning it is polynomial on each coset of some sublattice of finite index in $\bz^n$) and can be effectively calculated.  It contains deep information about the moduli space such as its orbifold Euler characteristic and intersection numbers of tautological classes on the compactified moduli space.  

In Section~\ref{sec:hur} we relate the count of Riemann surfaces $N_{g,n}(b_1,...,b_n)$ to a Hurwitz problem.  In a general Hurwitz problem one counts branched covers $\Sigma\to\Sigma'$ unramified outside fixed branch points $\{p_1,p_2,...,p_r,p_{r+1}\}\subset\Sigma'$ with partitions $\lambda^{(1)},\lambda^{(2)},...,\lambda^{(r+1)}$ describing the profile of the cover over the branch points.  We describe two Hurwitz problems---simple Hurwitz numbers and a Belyi Hurwitz problem.  The simple Hurwitz number $H_{g,\mu}$ counts connected genus $g$ covers of $S^2$ with ramification $\mu=(b_1,...,b_n)$ over $\infty$ and simple branching, i.e. ramification $(2,1,1,....)$, elsewhere.  The Belyi Hurwitz problem, which gives an equivalent description of $N_{g,n}(b_1,...,b_n)$, counts connected genus $g$ covers of $S^2$ unramified outside three branch points $\{p_1,p_2,p_3\}$ with profile $\lambda^{(3)}=(b_1,...,b_n)\in\bz_+^n$ over $p_3$, the partition  $\lambda^{(2)}=(2,2,...,2)$ over $p_2$,  and over $p_1$ all partitions containing no 1s.   

In Section~\ref{sec:lap} we describe the framework of Eynard and Orantin \cite{EOrInv} which gives a common context for different moduli space and Hurwitz problems.  In particular it explains the appearance of intersection numbers on the compactified moduli space of curves in each problem.\\

\noindent {\em Acknowledgements.}  The author would like to thank the referee for many useful comments.

\section{Cell decompositions of the moduli space of curves}

The aim of this section is to describe the cell decomposition of the {\em decorated} moduli space $\modm_{g,n}\times\br^n_+$ of genus $g$ curves with $n$ labeled points equipped with positive numbers $(b_1,...,b_n)\in\br_+^n$.  The moduli space has (real) dimension $6g-6+2n$ so the decorated moduli space has dimension $6g-6+3n$.  The cell decomposition of $\modm_{g,n}\times\br^n_+$ induces many cell decompositions on the moduli space $\modm_{g,n}$.   

The decorated moduli space has a cell decomposition  
\begin{equation}  \label{eq:cell}
\modm_{g,n}\times\br^n_+\cong\left(\bigsqcup_{\Gamma\in \fat}P_{\Gamma}\right)/\sim
\end{equation}
where the indexing set $\fat$ is the set of labeled fatgraphs of genus $g$ and $n$ boundary components, described below, and the cell $P_{\Gamma}\cong\br_+^{e(\Gamma)}$ for $e(\Gamma)$ the number of edges of the graph $\Gamma$.   This was proved independently by Harer \cite{HarVir}, who presented a proof of Mumford using Strebel differentials, and by Penner \cite{PenDec} using hyperbolic geometry.
\begin{ex}
The moduli space $\modm_{0,4}$ parametrises configurations of four distinct labeled points on $S^2$ up to conformal equivalence.  The group $PSL(2,\bc)$ acts conformally on $S^2$ and is 3-transitive, i.e. it can take any three points to any three points in $S^2$.  Hence 
\[\modm_{0,4}\cong S^2-\{0,1,\infty\}\] 
where the labeled points 1, 2 and 3 are taken to 0, 1 and $\infty$, respectively, and the labeled point 4 is taken to a point of $S^2-\{0,1,\infty\}$.   The decorated moduli space $\modm_{0,4}\times\br^4_+$ is 6-dimensional.  The cell decomposition of the decorated moduli space for $(g,n)=(0,4)$ consists of
\[ \modm_{0,4}\times\br^n_+\cong \{64\ [6{\rm -cells}],\ 144\ [5{\rm -cells}],\ 99\ [4{\rm -cells}],\ 20\ [3{\rm -cells}]\}/\sim \] 
and as expected
\[ 64-144+99-20=-1=\chi\left(\modm_{0,4}\right)=\chi\left(S^2-\{0,1,\infty\}\right).\]  
\end{ex}
The example uses the 327 labeled fatgraphs of genus 0 with 4 boundary components (produced from 21 unlabeled fatgraphs) which we now describe.

\subsection{Fatgraphs}  \label{sec:fat}

\begin{defn}
A fatgraph is a connected graph $\Gamma$ with all vertices of valency $>2$ endowed with a cyclic ordering of half-edges at each vertex.  It is uniquely determined by the triple $(X,\tau_0,\tau_1)$ where $X$ is the set of half-edges of $\Gamma$---so each edge of $\Gamma$ appears in $X$ twice---$\tau_1:X\to X$ is the involution that swaps the two half-edges of each edge and $\tau_0:X\to X$ the automorphism that permutes cyclically the half-edges with a common vertex. The underlying graph $\Gamma$ has vertices $X_0=X/\tau_0$, edges $X_1=X/\tau_1$ and boundary components $X_2=X/\tau_2$ for $\tau_2=\tau_0\tau_1$. 
\end{defn}
 One can allow valence $\leq 2$ vertices in the definition of fatgraphs.  Nevertheless, we find it more useful here to exclude such vertices. 
 \begin{figure}[ht]  
	\centerline{\includegraphics[height=2.5cm]{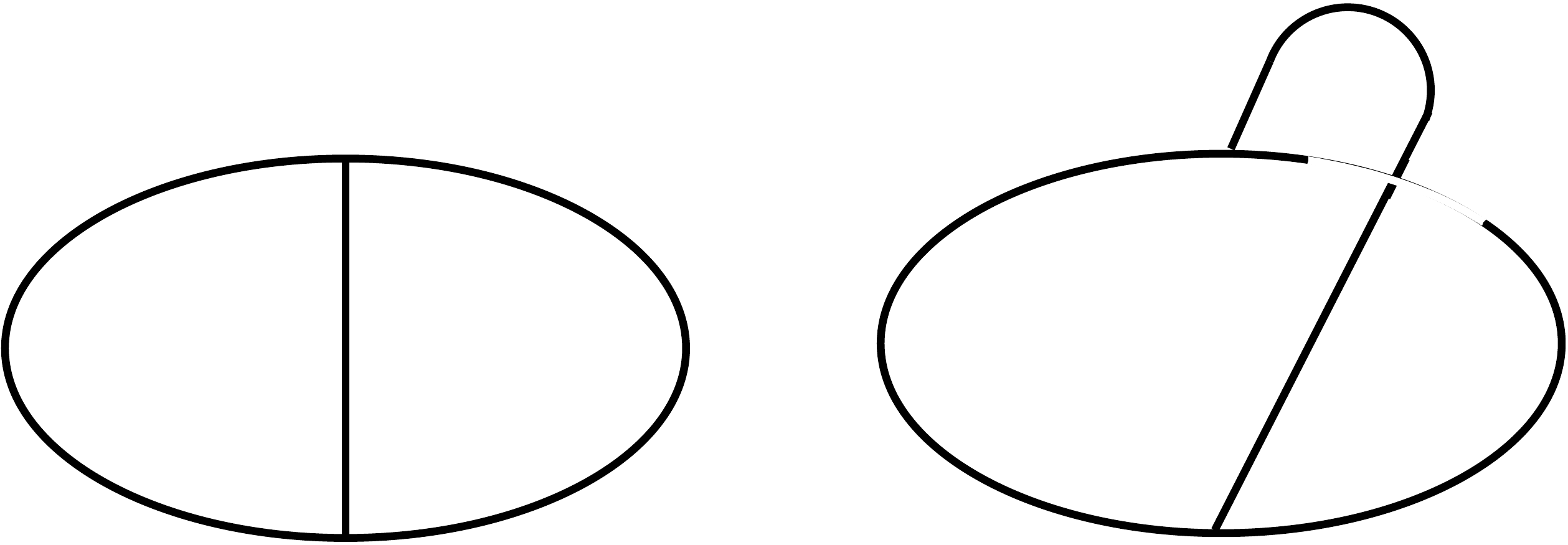}}
	\caption{Fatgraphs}
	\label{fig:fat}
\end{figure}

A fatgraph structure allows one to uniquely thicken the graph to a surface with boundary or equivalently it is an embedding up to isotopy of $\Gamma$ into an orientable surface with disk complements containing labeled points. In particular it acquires a type $(g,n)$ for $g$ the genus and $n$ the number of boundary components.  The two fatgraphs in Figure~\ref{fig:fat} are different, although the underlying graphs are the same.  
They have genus 0 and 1 which is made clear in Figure~\ref{fig:fats}.  The cyclic ordering of the half-edges with a common vertex is induced by the orientation of the page.
\begin{figure}[ht]  
	\centerline{\includegraphics[height=2.8cm]{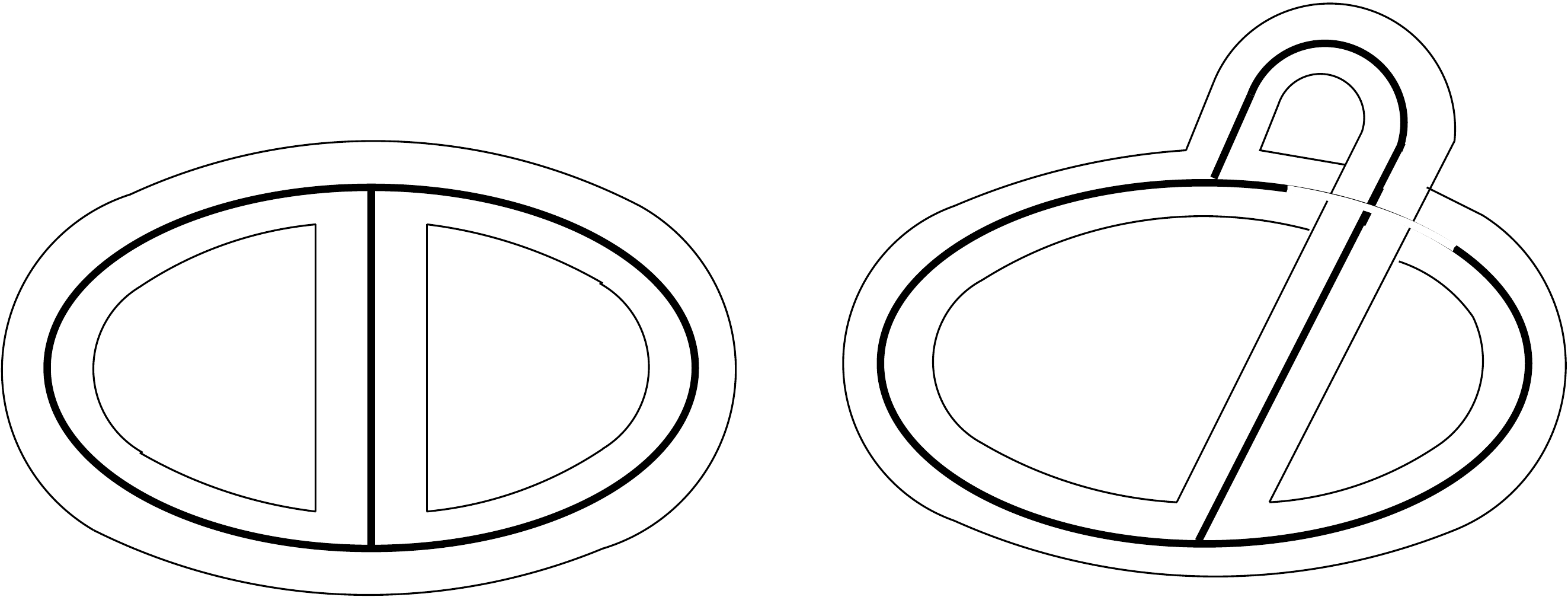}}
	\caption{Graphs embedded in genus 0 and 1 surfaces}
	\label{fig:fats}
\end{figure}

An {\em automorphism} of a fatgraph $\Gamma$ is a permutation $\phi:X\to X$ that commutes with $\tau_0$ and $\tau_1$.  An automorphism descends to an automorphism of the underlying graph. The group generated by $\tau_0$ and $\tau_1$ acts transitively on $X$, so an automorphism that fixes an oriented edge is necessarily trivial since $\phi(E)=E$ implies $\phi(\tau_0 E)=\tau_0 E$ and $\phi(\tau_1 E)=\tau_1 E$.

A {\em labeled} fatgraph is a fatgraph with boundary components labeled $1,...,n$.  An {\em automorphism} of a labeled fatgraph $\Gamma$ is a permutation $\phi:X\to X$ that commutes with $\tau_0$ and $\tau_1$ and acts trivially on $X_2$. 
\begin{defn}
Define $\fat$ to be the set of all labeled fatgraphs of genus $g$ and $n$ boundary components.  
\end{defn}
If $\Gamma\in \fat$ then the valency $>2$ condition on the vertices ensures that the number of edges $e(\Gamma)$ of $\Gamma$ is bounded $e(\Gamma)\leq 6g-6+3n$ with equality when the graph is trivalent.  In particular, $\fat$ is a finite set.  For small examples, $\fato_{0,3}$ consists of 4 labeled fatgraphs (from 2 unlabeled fatgraphs), $\fato_{1,1}$ consists of 2 labeled fatgraphs and $\fato_{0,4}$ consists of 327 labeled fatgraphs (from 21 unlabeled fatgraphs.)

\begin{ex}
[Calculations of automorphism groups of fatgraphs.]  

Any automorphism of the graph in $\fato_{0,3}$ in Figure~\ref{fig:fataut} must fix the boundary components 1 and 2, say, and hence it must fix the oriented edge  between them.  By the remark above, this implies the automorphism is trivial and hence the automorphism group is trivial.  This argument generalises to any fatgraph in $\fato_{0,n}$ since there is always an edge common to two different boundary components.

There are two fatgraphs in $\fato_{1,1}$ shown in Figure~\ref{fig:fataut} .  In both these examples, $\tau_0$ and $\tau_1$ commute so they are in fact automorphisms of the underlying graph.  In both cases the automorphism group is cyclically generated by $\tau_0\tau_1$ yielding $\bz_6$ and $\bz_4$ respectively.
\begin{figure}[ht]  
	\centerline{\includegraphics[height=6cm]{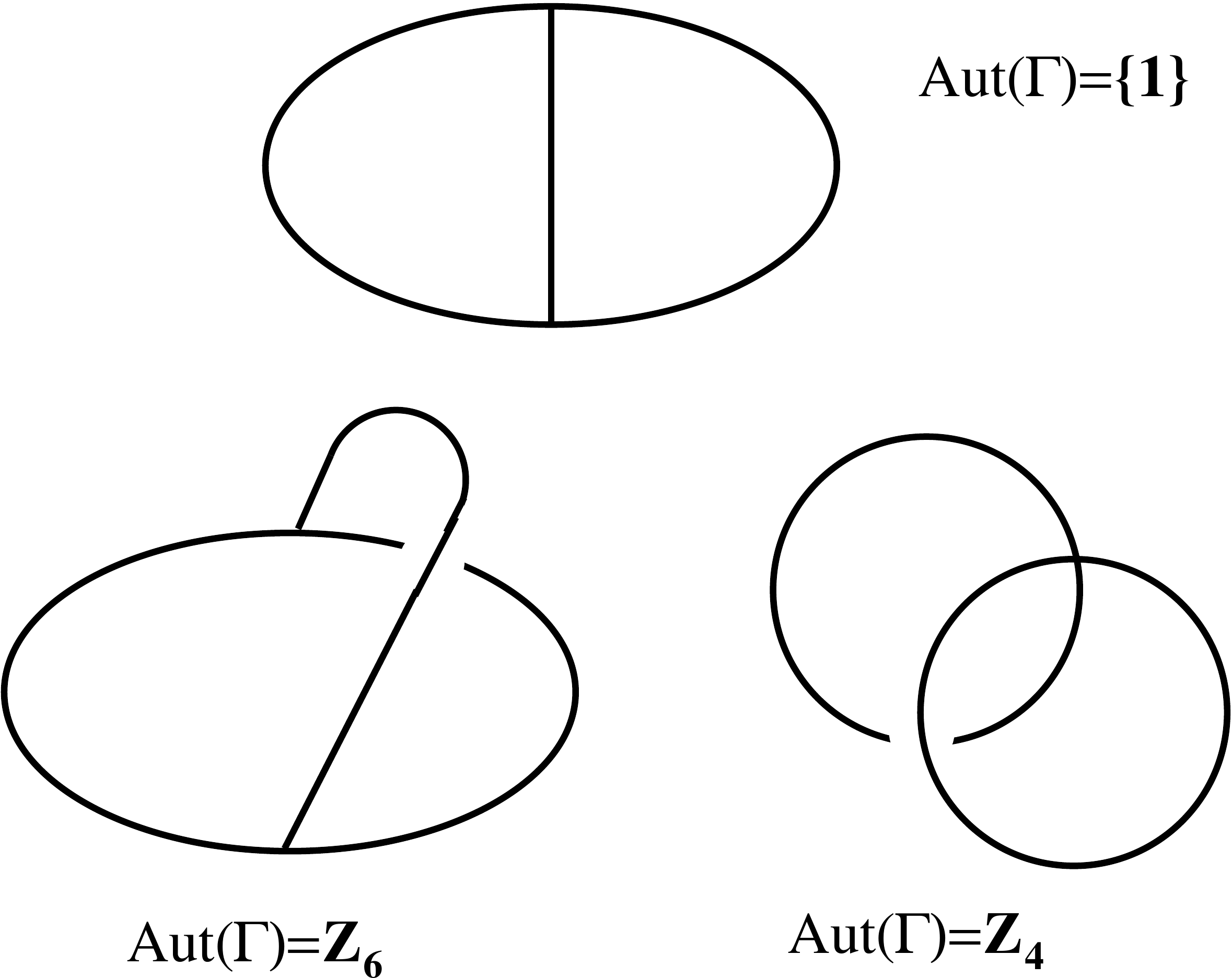}}
	\caption{Automorphism groups of fatgraphs}
	\label{fig:fataut}
\end{figure}
\end{ex}
A metric on a labeled fatgraph $\Gamma$ assigns positive numbers---lengths---to each edge of the fatgraph.   
\begin{defn}   
Define $P_{\Gamma}\cong R_+^{e(\Gamma)}$ to consist of all metrics on the fatgraph $\Gamma$. 
\end{defn}
The cells $P_{\Gamma}$ have maximum dimension $6g-6+3n$ when $\Gamma$ is trivalent.  They naturally glue together along lower dimensional cells by identifying the limiting metric on a fatgraph when the length of an edge $l_E\to 0$ with the metric on the new fatgraph with the edge $E$ contracted.  This yields the the cell-complex
\[ \modm_{g,n}^{\rm combinatorial}:=\left(\bigsqcup_{\Gamma\in \fat}P_{\Gamma}\right)/\sim\]
where $\sim$ denotes gluing lower dimensional cells and identifying isometric fatgraphs.  In particular automorphisms of a fatgraph lead to orbifold points.

The cell decomposition (\ref{eq:cell}) is equivalent to the claim that the decorated moduli space is homeomorphic to this cell complex
\[\modm_{g,n}\times\br_+^n\cong\modm_{g,n}^{\rm combinatorial}.\]
There are two {\em different} homeomorphisms yielding two proofs of (\ref{eq:cell}).
Both use special metrics in the given conformal class on the complement of the labeled points---a singular flat metric and a hyperbolic metric.  The first proof is due to Harer and Mumford, \cite{HarVir}.   It begins with the simple fact that any disk with a marked point is conformally equivalent to the unit disk $\{|z|<1\}$, and the complement of 0 in the closed unit disk is conformally equivalent to a half infinite cylinder:
\begin{align*}
\{z:0<|z|\leq 1\}&\simeq [0,\infty)\times S^1\\
|dz|^2&\simeq \frac{|dz|^2}{|z|^2}=dt^2+d\theta^2,\quad \ln z=t+i\theta.
\end{align*}
The half infinite cylinder with metric $|dz|^2/|z|^2$ is conformally equivalent to a half infinite cylinder with circumference $a$ by rescaling to get $a^2|dz|^2/|z|^2$.  The local metric $a^2|dz|^2/|z|^2$ is the norm of the local meromorphic quadratic differential $a^2(dz)^2/z^2$.  The global version of this is a meromorphic quadratic differential on a compact Riemann surface locally equivalent to $a^2(dz)^2/z^2$ at each labeled point.  The coefficient $a^2$ is known as the residue of the quadratic differential with pole of order 2 since it is independent of the local parametrisation---if $w$ is another local parameter so that $z=z(w)$ with $z(0)=0$ and $z'(0)\neq 0$ then $a^2(dz)^2/z^2=a^2(dw)^2/w^2\cdot w^2z'(w)^2/z(w)^2=a^2(dw)^2/w^2\cdot(1+wh(w))$ for $h$ analytic at $w=0$.  Hence we can equivalently state that the meromorphic quadratic differential has a pole of order two with prescribed residue at each labeled point.
It defines two foliations, one having compact leaves, where the leaves correspond to the geodesics $t=$ constant and $\theta=$ constant in the coordinates above.  More invariantly the leaves of the foliation with compact, respectively non-compact, leaves are the paths along which the meromorphic quadratic differential is negative, respectively positive.  Such a meromorphic quadratic differentials exists and is unique and is known as a {\em Strebel differential} \cite{StrQua}.  Zeros of the quadratic differential correspond to cone points which are multiples of $\pi$ in the flat metric.  Thus, in a given conformal class together with a positive number $b_i$ assigned to each labeled point, the norm of the unique Strebel differential is a singular flat metric with cone points of cone angles multiples of $\pi$ such that a deleted neighbourhood of each labeled point is a flat cylinder $S^1_{b_i}\times[0,\infty)$ of circumference $b_i$.  The flat cylinders meet along a fatgraph $\Gamma$ as in figure~\ref{fig:flat}. The fatgraph inherits lengths  on its edges and hence it is a metric fatgraph or equivalently an element of $P_{\Gamma}$ thus proving one direction of (\ref{eq:cell}).  The converse, that a labeled metric fatgraph produces a decorated Riemann surface, is obtained by gluing together cylinders along a fatgraph.  
\begin{figure}[ht]  
	\centerline{\includegraphics[height=6cm]{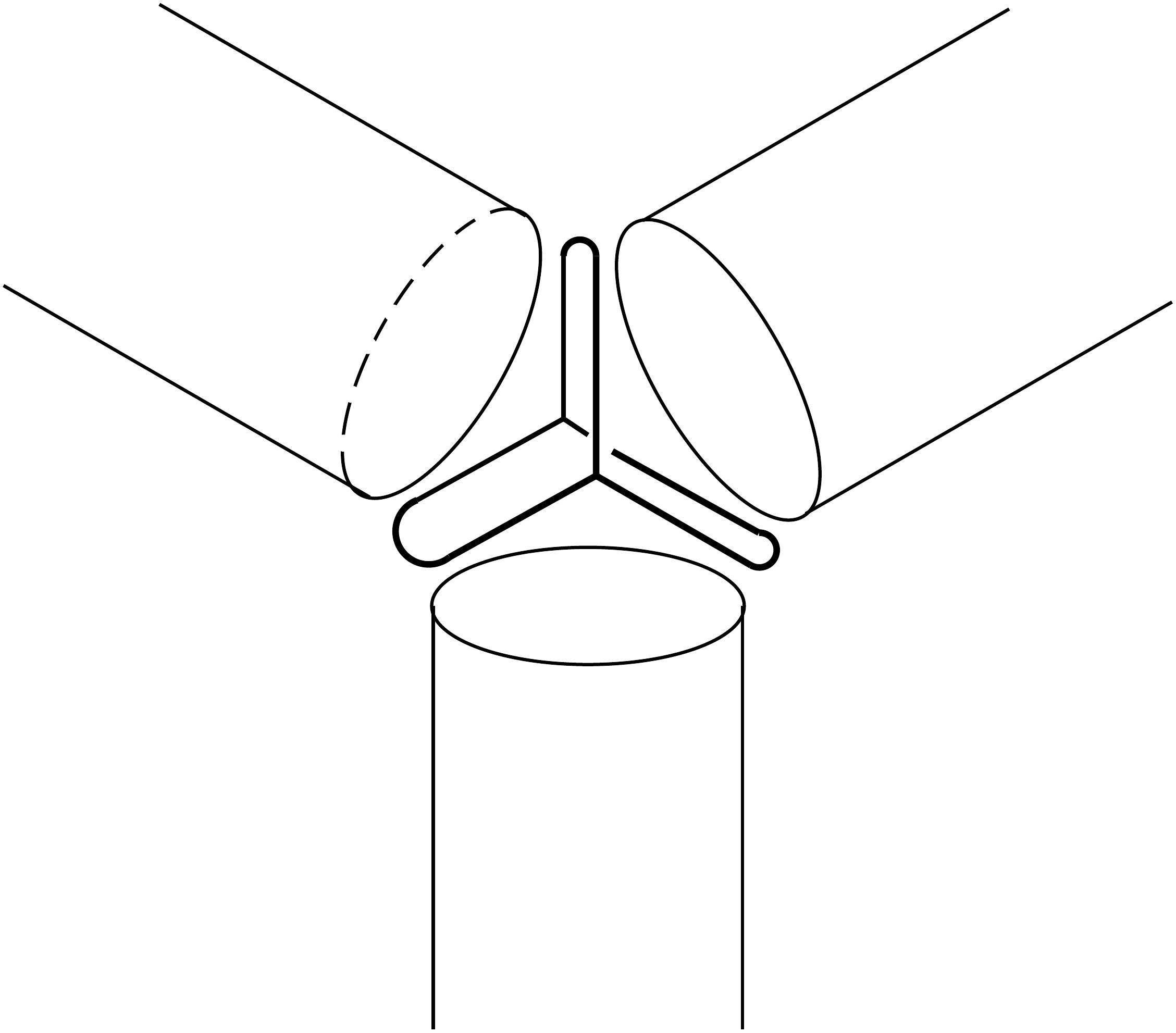}}
	\caption{Decomposition of a conformal surface into flat cylinders}
	\label{fig:flat}
\end{figure}
Section~\ref{sec:belyi} describes a special case of gluing together cylinders along a fatgraph which is quite explicit and central to this paper.  This uses the conformal equivalence between the half infinite cylinder and the set $\bc-[0,1]$ together with covers of this cylinder. 

The second cell decomposition is due to Penner.  It uses the identification of the moduli space with the moduli space of genus $g$ oriented hyperbolic surfaces with $n$ labeled cusps.  The decorated moduli space consists of oriented hyperbolic surfaces together with a choice of horocycle at each cusp of the hyperbolic surface \cite{PenDec}.  The {\em cut locus} of the horocycles, that is the set of points with more than one shortest path to the set of horocycles, defines the metric fatgraph hence an element of $P_{\Gamma}$ giving one direction of (\ref{eq:cell}).  Similarly to the first proof, the converse is obtained by gluing together hyperbolic pieces.  See for example, \cite{BEpNat}.  A related approach is to identify the decorated moduli space with the moduli space of genus $g$ oriented hyperbolic surfaces with $n$ labeled geodesic boundary components of specified length.  This is a symplectic deformation of the usual moduli space and in particular, homeomorphic.  The fatgraphs then arise as the cut loci of the geodesic boundary \cite{DoInt}.\\

For each $(b_1,...,b_n)\in\br_+^n$, the natural projection
\[\pi:\modm_{g,n}^{\rm combinatorial}\cong\modm_{g,n}\times\br_+^n\to\br_+^n\]
has fibres homeomorphic to the moduli space 
\[\modm_{g,n}^{\rm combinatorial}(b_1,...,b_n):=\pi^{-1}(b_1,...,b_n)\cong\modm_{g,n}.\]  
\begin{defn}   \label{def:cell}
Define the set of all metrics on the fatgraph $\Gamma$ with fixed boundary lengths ${\bf b}=(b_1,...,b_n)\in\br_+^n$ to be
\[P_{\Gamma}(b_1,...,b_n):=P_{\Gamma}\cap\pi^{-1}(b_1,...,b_n).\]  
\end{defn}
In particular the induced cell decomposition on each fibre
\[\modm_{g,n}^{\rm combinatorial}(b_1,...,b_n)=\left(\bigsqcup_{\Gamma\in \fat}P_{\Gamma}(b_1,...,b_n)\right)/\sim\]
defines a cell decomposition of $\modm_{g,n}$ depending on $(b_1,...,b_n)$.
\addtocounter{ex}{-2}
\begin{ex}[{\bf continued}]
For any given $(b_1,b_2,b_3,b_4)$, the intersection of most of the 327 cells of $\modm_{0,4}^{\rm combinatorial}$ with $\modm_{0,4}^{\rm combinatorial}(b_1,b_2,b_3,b_4)$ is empty making the induced cell decomposition on $\modm_{0,4}$, which is a triangulation by convex polygons, simpler.  
\begin{table}[ht]  \label{tab:triang}
\caption{Induced triangulations on $\modm_{0,4}$}
\begin{spacing}{1.4}  
\begin{tabular}{||c||c|c|c|c||} 
\hline\hline

$(b_1,b_2,b_3,b_4)$ &0-cells&1-cells&2-cells&polygons\\ \hline

$b_1=b_2=b_3=b_4$&0&3&2&2\ triangles\\ \hline
$b_1>>b_2=b_3=b_4$&6&18&11&8 triangles, 3 squares\\ \hline
$b_1>>b_2>>b_3>>b_4$&9&24&14&8 triangles, 6 squares\\ 
\hline\hline
\end{tabular} 
\end{spacing}
\end{table}
The table gives induced triangulations of $S^2-\{0,1,\infty\}$, so for example
$\modm_{0,4}^{\rm combinatorial}(b,b,b,b)$ consists of two ideal triangles---each has a face and 3 edges but the vertices are missing---with paired edges glued.   Figure~\ref{fig:cell3} shows the case $b_1>>b_2=b_3=b_4$.  Note in each case that the Euler characteristic is -1 as expected.  
\begin{figure}[ht]  
	\centerline{\includegraphics[height=3cm]{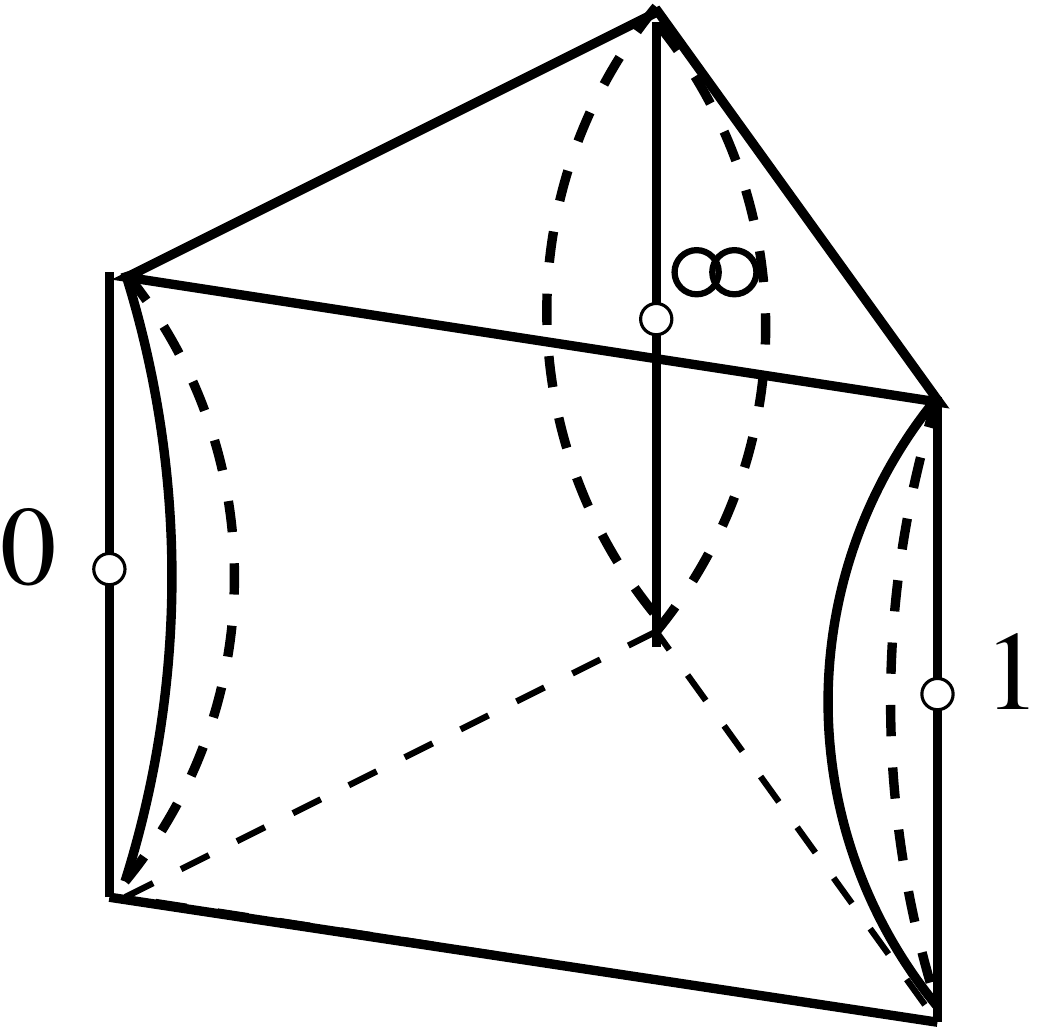}}
	\caption{Cell decomposition of $\modm_{0,4}$}
	\label{fig:cell3}
\end{figure}
\end{ex}
\addtocounter{ex}{+1}
\section{Convex polytopes, volumes and lattice points}  \label{sec:convpvl}

A {\em convex polytope} $P\subset\br^N$ is a bounded convex set whose closure is the convex hull of a finite set of vertices in $\br^N$.  Given a linear map $A:\br^N\to\br^n$ and ${\bf b}\in\br^n$ 
\[ P_A({\bf b}):=\{{\bf x}\in\br_+^N\ |\ A{\bf x}={\bf b}\}\]
defines a convex set and it is a convex polytope if $A$ has non-negative entries and non-zero columns.  
There is a natural volume form on such a convex polytope given by the quotient volume of the Euclidean volumes on $\br^N$ and $\br^n$.  Denote by $V_{P_A}({\bf b})$ the volume of $P_A({\bf b})$ with respect to this quotient volume.   The quotient volume satisfies 
\begin{equation}  \label{eq:qvol}
dV_{P_A}(b_1,...,b_n)db_1...db_n=dx_1...dx_N
\end{equation}
and is homogeneous in $b_1,...,b_n$.   If $A$ has rank $n$ then the quotient volume can be expressed in the following way.  Choose $n$ linearly independent columns of $A$ indexed by $\alpha=\{\alpha_1,...,\alpha_n\}\subset\{1,2.,...,N\}$ and form the submatrix $A_{\alpha}$.  The $N-n$ variables $\{x_i:i\notin\alpha\}$ parametrise the solution set $P_A$.  With respect to these variables the quotient volume is $\wedge_{i\notin\alpha}dx_i/\det{A_{\alpha}}$.  The quotient volume is not the same as the volume induced by the embedding.
\begin{ex}
\[A=\left(\begin{array}{ccc}1&2&2\\1&0&0\end{array}\right)\]
\begin{align*}P_A({\bf b})=\{{\bf x}\in\br_+^3|A{\bf x}={\bf b}\}\quad&\Rightarrow\quad x_1+2x_2+2x_3=b_1,\quad x_1=b_2,\quad x_i>0\\&\Rightarrow\quad 2x_2+2x_3=b_1-b_2 
\end{align*}
The quotient volume $dv$ is determined by 
\[dv\wedge d(x_1+2x_2+2x_3)\wedge dx_1=dx_1\wedge dx_2\wedge dx_3\quad\Rightarrow\quad dv=\frac{dx_2}{2}\left(=-\frac{dx_3}{2}\right)\]
and if $b_1-b_2>0$ then $x_2\in\left(0,\frac{1}{2}(b_1-b_2)\right)$, thus
\[ V_{P_A}(b_1,b_2)=\left\{\begin{array}{cl}0,&b_1<b_2\\
\frac{1}{4}(b_1-b_2),&b_1>b_2\end{array}\right.\]
\end{ex}

In general the volume of a convex polytope is difficult to calculate.  It is piecewise polynomial in the $b_i$ so in particular piecewise continuous.   The polynomials depend on chambers defined as the complement of hypersurfaces in the $b_i$.  Nevertheless, the Laplace transform of $V_{P_A}({\bf b})$ with respect to ${\bf b}$ is easy to calculate.  Denote the columns of the matrix $A$ by $\alpha_1,...,\alpha_N\in\br^n$ and assume all of its entries are nonnegative so that the Laplace transform is defined.   Using (\ref{eq:qvol}), the Laplace transform of $V_{P_A}({\bf b})$ is 
\begin{align*}
\widehat{V}_{P_A}({\bf s})&=\int \exp(-{\bf b}\cdot{\bf s})V_{P_A}(b_1,...,b_n)db_1...db_n\\
&=\int \exp(-{\bf b}\cdot{\bf s})\left\{\int dV_{P_A}(b_1,...,b_n)\right\}db_1...db_n\\
&=\int \exp(-{\bf b}\cdot{\bf s})dx_1...dx_N\\
&=\int\exp\left(-\sum_{i=1}^N(\alpha_i\cdot{\bf s})x_i\right)dx_1...dx_N\\
&=\prod_i\frac{1}{\alpha_i\cdot{\bf s}}.
\end{align*}
\addtocounter{ex}{-1}
\begin{ex}[{\bf continued}]
\[A=\left(\begin{array}{ccc}1&2&2\\1&0&0\end{array}\right)\]
\[\Rightarrow\widehat{V}_{P_A}(s_1,s_2)=\frac{1}{(s_1+s_2)2s_12s_1}\]
which can be directly checked to be the Laplace transform of 
\[ V_{P_A}(b_1,b_2)=\left\{\begin{array}{cl}0,&b_1<b_2\\
\frac{1}{4}(b_1-b_2),&b_1>b_2.\end{array}\right.\]
\end{ex}
As demonstrated in the example, the Laplace transform $\widehat{V}_{P_A}({\bf s})$ is homogeneous due to the homogeneity of $V_{P_A}({\bf b})$, and its poles away from $s_i=0$ reflect the lack of continuity of $V_{P_A}({\bf b})$.
\subsection{Lattice points}
If the matrix $A$ is defined over the integers i.e. $A:\bz^N\to\bz^n$ (and $A$ has non-negative entries and non-zero columns), then for $b\in\bz^n$ 
\[ P_A({\bf b}):=\{{\bf x}\in\br_+^N\ |\ A{\bf x}={\bf b}\}\]
is a {\em rational} convex polytope meaning its vertices lie in $\bq^N$.   One can count integral solutions to $A{\bf x}={\bf b}$.  Define
\[N_{P_A}({\bf b}):=\#\{P_A({\bf b})\cap\bz_+^N\}.\]
This is also known as a {\em vector partition function} \cite{StuVec}.  Denote the columns of $A$ by $\alpha_i\in\bz^n$, $i=1,...,N$ and for any ${\bf v}=(v_1,...,v_n)\in\bz^n$ let ${\bf z}^{\bf v}=\prod_iz_i^{v_i}$
The discrete Laplace transform of $N_{P_A}({\bf b})$ is 
\begin{align*}
\widehat{N}_{P_A}({\bf z})&=\sum_{b_i>0} N_{P_A}(b_1,...,b_n)z_1^{b_1}...z_n^{b_n}\\
&=\hspace{-4mm}\sum_{\{x>0:Ax=b\}}\hspace{-4mm} z_1^{b_1}...z_n^{b_n}\\
&=\prod_{i=1}^N\sum_{x_i>0}\left({\bf z}^{\alpha_i}\right)^{x_i}\\
&=\prod_{i=1}^N\frac{{\bf z}^{\alpha_i}}{1-{\bf z}^{\alpha_i}}.
\end{align*}
If one sets $z_i=\exp(-s_i)$ then the discrete Laplace transform looks like a discrete version of the Laplace transform.  The discrete Laplace transform of a polynomial has poles only at $z_i=1$ so the poles at, say $z_i=1/z_j$ reflect the piecewise behaviour and poles at $z_i=-1$ represent quasipolynomial behaviour.
\addtocounter{ex}{-1}
\begin{ex}[{\bf continued}]
\[A=\left(\begin{array}{ccc}1&2&2\\1&0&0\end{array}\right)\]
\begin{eqnarray*}\Rightarrow& N_{P_A}(b_1,b_2)&\equiv 0,\quad b_1-b_2{\rm\ odd}\\
&N_{P_A}(b_1,b_2)&=\left\{\begin{array}{cl}0,&b_1\leq b_2\\
\frac{1}{2}(b_1-b_2)-1,&b_1>b_2\end{array}\right.,\quad b_1-b_2{\rm\ even.}
\end{eqnarray*}
So $N_{P_A}(b_1,b_2)$ is a piecewise defined quasi-polynomial in $\bf b$.  The discrete Laplace transform 
\[ \widehat{N}_{P_A}(z_1,z_2)=\frac{z_1^5z_2}{(1-z_1z_2)(1-z_1^2)^2}\]
is calculated using the columns of the matrix.  It  can be checked by explicitly summing $N_{P_A}(b_1,b_2)$ or more conveniently by checking only the first few terms in a Taylor series around $z_i=0$ using a computer.
\end{ex}
A function of integers $F(a_1,...,a_n)$ is a {\em quasi-polynomial}  if it decomposes into a collection of polynomials that depend on some modular information about the integers $(a_1,...,a_n)$.  More precisely, a quasi-polynomial coincides with a polynomial on each coset of some sublattice of finite index in $\bz^n$.  For example, $F(a)=(1+(-1)^a)\cdot a$ is quasi-polynomial because it decomposes into the two polynomials $2a$, respectively $0$, when $a$ is even, respectively odd.  For a general $A:\bz^N\to\bz^n$, $N_{P_A}(b_1,...,b_n)$ is a {\em piecewise} defined quasi-polynomial in ${\bf b}\in\bz^n$.   It is polynomial on each coset of the sublattice $A\bz^N\subset\bz^n$ where it resembles the volume $V_{P_A}(b_1,...,b_n)$ although it is no longer homogeneous.

The number of lattice points in a polytope is intuitively an approximation for its volume.  The following theorem gives the simplest relationship between number of lattice points and volume.
\begin{thm}[Ehrhart \cite{EhrPol}]  \label{th:ehrhart}
If $P\subset\br^n$ is a dimension $n$ convex polytope with integral vertices and interior $P^0$ define 
\[N_P:=\#\{P\cap\bz^n\},\quad N_P(k):=\#\{kP\cap\bz^n\}.\]
Then $N_P(k)$ is a degree $n$ polynomial in $k$
\[ N_P(k)={\rm Vol}(P)k^n+...+1\]
and
\[N_{P^0}(k)=(-1)^nN_P(-k)={\rm Vol}(P)k^n+...+(-1)^n.\] 
\end{thm}
The constant term gives the Euler characteristic $\chi(P)$, respectively   $\chi(P^0)$, (where $\chi[0,1]=1$ and $\chi(0,1)=-1$.)
\begin{ex}
\[ P=\begin{array}[b]{cc}\bullet&\\\bullet&\bullet\end{array},\quad
     kP=\begin{array}[b]{cccc}\bullet&&&\\\bullet&\bullet&&\\\bullet&\bullet&\bullet&\\
     \bullet&\bullet&\bullet&\bullet\end{array}\]
     \[N_P(k)=\frac{1}{2}(k+1)(k+2)\]
\end{ex}
A similar relationship holds between the number of lattice points and quotient volume of a convex polytope $P_A({\bf b}):=\{{\bf x}\in\br_+^N\ |\ A{\bf x}={\bf b}\}$  defined by a linear map of full rank $A:\bz^N\to\bz^n$ and $b\in\bz^n$.  If $A$ is not surjective then there exists a rational matrix $B$ that defines an isomorphism from $A\cdot\bz^N$ onto $\bz^n$ and we define ${\rm ind}_A=1/\det B$.  More invariantly:
\begin{defn}
For an integer matrix $A:\bz^N\to\bz^n$ define 
\[{\rm ind}_A={\rm\ index\ of\ the\ sublattice\ }A\cdot\bz^N\subset\bz^n.\]
\end{defn}
From analysing the Laplace transform or by replacing $A$ with $BA$ so that ${\rm ind}_{BA}=1$ we get the analogue of Theorem~\ref{th:ehrhart}.  For $b\in A\cdot\bz^N$,
\[ N_{P_A}({\bf b})={\rm ind}_A\cdot V_{P_A}({\bf b})+{\rm\ lower\ degree\ terms.}\]
In general the constant term is not $\pm 1$ as is shown in the next example.
\begin{ex}  \label{ex:notpm1}
\[A=\left(\begin{array}{cccc}1&1&2&0\\1&1&0&2\end{array}\right)\]
For $b_1$ and $b_2$ odd,
\[N_{P_A}(b_1,b_2)=\left\{\begin{array}{cl}\frac{1}{4}b_1^2-b_1+\frac{3}{4},&\quad b_1\leq b_2\\
\frac{1}{4}b_2^2-b_2+\frac{3}{4},&\quad b_1\geq b_2.\end{array}\right.\]
\end{ex}

\subsection{Convex polytopes from fatgraphs}

Given a fatgraph $\Gamma$ its {\em incidence matrix}  $A_{\Gamma}$ is defined by:
\begin{align*}
A_{\Gamma}\quad:\quad\br^{e(\Gamma)}&\to\quad\br^n\\
{\rm edge\ }&\mapsto {\rm\ incident\ boundary\ components.}
\end{align*}

\begin{ex}
Let $\Gamma$ and $\Gamma'$ be the genus 0 and genus 1 fatgraphs in the diagram below.  Then their respective incidence matrices are
\[ A_{\Gamma}=\left[\begin{array}{ccc}1&1&0\\1&0&1\\0&1&1\end{array}\right],\quad A_{\Gamma'}=[2\quad 2\quad 2]\]
\begin{figure}[ht]  
	\centerline{\includegraphics[height=2.5cm]{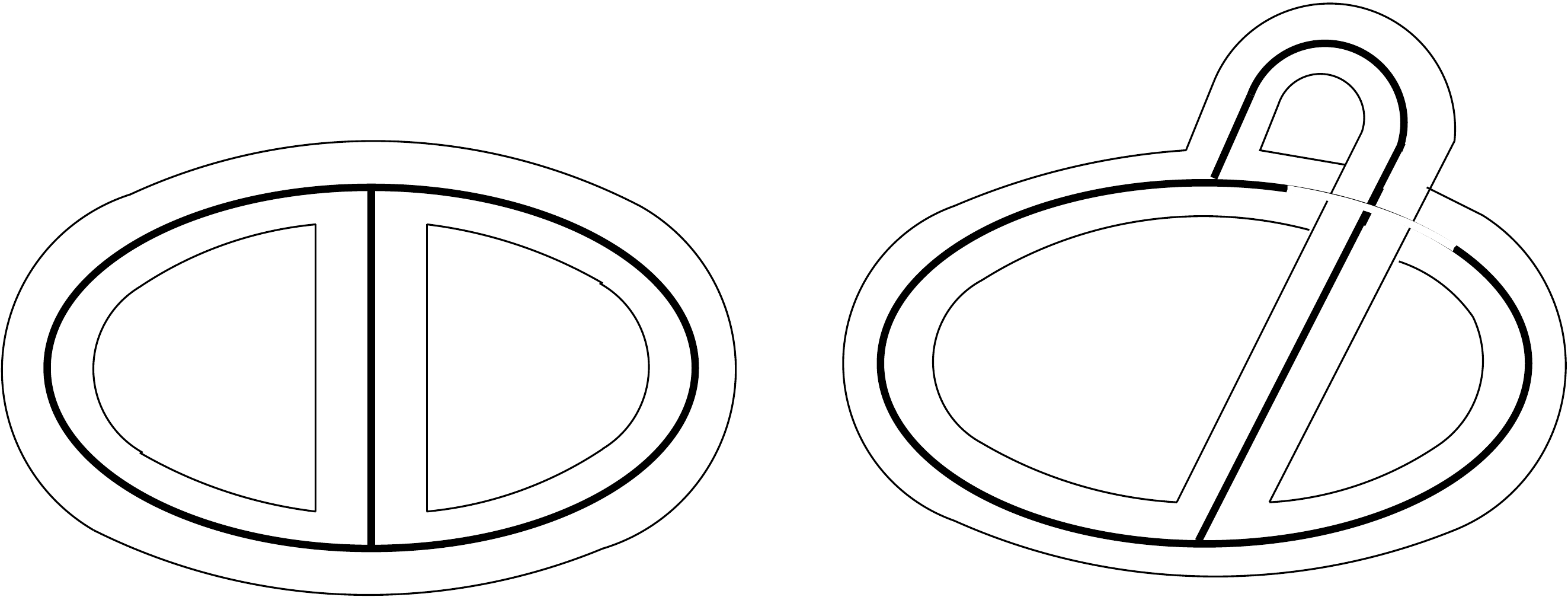}}
\end{figure}
\end{ex}
The cell decomposition of $\modm_{g,n}^{\rm combinatorial}(b_1,...,b_n)$ which induces a cell decomposition of $\modm_{g,n}$ uses cells defined in Definition~\ref{def:cell} which consist of metric labeled fatgraphs with boundary lengths $(b_1,...,b_n)$.  But these are just the convex polytopes defined by the incidence matrices of fatgraphs.  In other words
\[P_{\Gamma}(b_1,...,b_n)=P_{A_{\Gamma}}(b_1,...,b_n)=\{{\bf x}\in\br_+^{e(\Gamma)}|A_{\Gamma}{\bf x}={\bf b}\}.\]  
Each $P_{\Gamma}(b_1,...,b_n)$ is a convex polytope which is the interior of a compact convex polytope.  The closure of $P_{\Gamma}(b_1,...,b_n)$ consists of solutions $A_{\Gamma}{\bf x}={\bf b}$ where some $x_i=0$.  These solutions can be identified with points of $P_{\Gamma'}(b_1,...,b_n)$ where the fatgraph $\Gamma'$ is obtained from $\Gamma$ by contracting edges, together with points that cannot be identified with a fatgraph, but instead correspond to points in the compactification of the moduli space \cite{DNoCou}.

Define the volume of the cell $P_{\Gamma}$ using the quotient volume $V_{\Gamma}=V_{P_{A_{\Gamma}}}$.  Kontsevich \cite{KonInt} defined the total volume of $\modm_{g,n}^{\rm combinatorial}(b_1,...,b_n)$
\begin{defn}
$V_{g,n}(b_1,...,b_n)=\displaystyle\sum_{\Gamma\in \fat}\frac{1}{|Aut \Gamma|}V_{\Gamma}(b_1,...,b_n).$
\end{defn}
\begin{table}[ht]  \label{tab:vol}
\caption{Kontsevich volumes}
\begin{spacing}{1.4}  
\begin{tabular}{||l|c|c||} 
\hline\hline

{\bf g} &{\bf n}&$V_{g,n}(b_1,...,b_n)$\\ \hline

0&3&$\frac{1}{2}$\\ \hline
1&1&$\frac{1}{96}b_1^2$\\ \hline
0&4&$\frac{1}{8}\left(b_1^2+b_2^2+b_3^2+b_4^2\right)$\\ \hline
1&2&$\frac{1}{2^83}\left(b_1^2+b_2^2\right)^2$\\ \hline
2&1&$\frac{1}{2^{17}3^3}b_1^8$\\
\hline\hline
\end{tabular} 
\end{spacing}
\end{table}

The quotient volume yields $V_{0,3}(b_1,b_2,b_3)=1/2$ reflecting the fact that the cell decompositions of $\modm_{0,3}$, which is just  single point, use matrices of determinant 2.  Only the trivalent fatgraphs contribute to the sum.  The Laplace transform of $V_{g,n}$ appears as $I_g$ in \cite{KonInt}.  It is a non-trivial fact for $n>1$ that $V_{g,n}(b_1,...,b_n)$ is actually a polynomial in the $b_i$.  It is a sum of the piecewise polynomials $V_{\Gamma}(b_1,...,b_n)$.  
\begin{lemma}  \label{th:chambers}
Each piecewise polynomial $V_{\Gamma}(b_1,...,b_n)$ is polynomial on each connected component, or chamber, of the complement of the $2^{n-1}$ hyperplanes
\[\{b_1\pm b_2\pm...\pm b_n=0\}.\]
\end{lemma}
\begin{proof}
The chambers are the interiors of maximal cones in the convex cone generated by the columns of $A_{\Gamma}$ \cite{BVeRes} for $\Gamma\in\fat$.  The hyperplanes are boundaries of the maximal cones which are generated by subsets of the columns of $A_{\Gamma}$.  Hence we need to prove that for a sub-fatgraph $\Gamma'\subset\Gamma$, if the columns of $A_{\Gamma'}$ span a hyperplane in $\br^n$ then that hyperplane is either a coordinate plane $b_j=0$ or $b_1\pm b_2\pm...\pm b_n=0$.  

Choose $\Gamma'$ to consist of exactly $n-1$ edges, and consider the two cases when $\Gamma'$ avoids a boundary component, and when $\Gamma'$ meets every boundary component of $\Gamma$.  (A boundary component of a fatgraph is an element of $X/\sigma_2$.)  In the first case, if the columns of $A_{\Gamma'}$ span a hyperplane in $\br^n$ and $\Gamma'$ avoids the $j$th boundary component, then the columns of $A_{\Gamma'}$ lie on the coordinate plane $b_j=0$.  In the second case, if the columns of $A_{\Gamma'}$ span a hyperplane in $\br^n$ and $\Gamma'$ meets all boundary components then the dual graph has $n$ vertices (dual to faces) and $n-1$ edges hence its Euler characteristic is 1 so the dual graph is a tree and, in particular, bipartite.  Choose the bipartite labeling of vertices to be $\pm 1$.  Equivalently, for $\Gamma'\hookrightarrow\Gamma\hookrightarrow\Sigma$ where $\Sigma$ is the surface such that $\Sigma-\Gamma=\sqcup_{j}^nD^2_j$ we have labeled connected regions of $\Sigma-\Gamma'$ with $\pm 1$ so that any edge has different labels on each side.  But then each column of $A_{\Gamma'}$ is orthogonal to the hyperplane $\pm b_1\pm b_2\pm...\pm b_n=0$ where $+$ or $-$ is chosen according to the bipartite labeling described above.
\end{proof}

An element of a non-empty chamber determines the truth of each inequality $b_1\pm b_2\pm...\pm b_n>0$.  Thus a non-empty chamber can be identified with a Boolean function of $n$ variables---the $n$ variables lie in the set $\{0,1\}$---constructed using only AND and OR operations, and such that swapping all 0s for 1s preserves the truth table.  The number of chambers for $n=1,2,3,4,5,...$ is 1,2,4,12,81,...  and appears as the Sloane sequence A001206.  

In fact Kontsevich defined two volumes on $\modm_{g,n}^{\rm combinatorial}(b_1,...,b_n)$.  He defined a second volume via a symplectic form on $\modm_{g,n}^{\rm combinatorial}(b_1,...,b_n)$ and proved that the symplectic volume is a constant multiple of the polytope volume.  He used this to claim that the coefficients of $V_{g,n}(b_1,...,b_n)$ give intersection numbers of Chern classes of the tautological line bundles $L_i$ over the compactified moduli space $\overline{\modm}_{g,n}$. 
\begin{thm}[\cite{KonInt}]  \label{th:Kon}
For $|{\bf d}|=\sum_id_i=3g-3+n$ and ${\bf d}!=\prod d_i!$ the coefficient $c_{\bf d}$ of $b^{2{\bf d}}=\prod b_i^{2d_i}$ in $V_{g,n}(b_1,...,b_n)$ is the intersection number
\[c_{\bf d}=\frac{1}{2^{5g-5+2n}{\bf d}!}\int_{\overline{\modm}_{g,n}}c_1(L_1)^{d_1}...c_1(L_n)^{d_n}.\]
\end{thm}
See \cite{LooCel} for a discussion of this result.

\subsection{Counting lattice points in the moduli space of curves}

For a fatgraph $\Gamma$ and integers $(b_1,...,b_n)$ define the number of lattice points in the rational convex polytope $P_{\Gamma}(b_1,...,b_n)$ by
\[ N_{\Gamma}(b_1,...,b_n):=N_{P_{\Gamma}}(b_1,...,b_n)=\#\{P_{\Gamma}({\bf b})\cap\bz_+^{e(\Gamma)}\}.\]
The number of lattice points $N_{\Gamma}(b_1,...,b_n)$ is piecewise quasi-polynomial and on each chamber, defined in Lemma~\ref{th:chambers},  it is quasi-polynomial.  The number of lattice points in $\modm_{g,n}^{\rm combinatorial}(b_1,...,b_n)$, defined in \cite{NorCou}, is the weighted sum of $N_{\Gamma}$ over all labeled fatgraphs of genus $g$ and $n$ boundary components:
\begin{defn} \label{th:lcp}
$N_{g,n}(b_1,...,b_n)=\displaystyle\sum_{\Gamma\in \fat}\frac{1}{|{\rm Aut\ } \Gamma|}N_{\Gamma}(b_1,...,b_n).$
\end{defn}
All fatgraphs contribute to the sum, unlike $V_{g,n}(b_1,...,b_n)$ where only the trivalent fatgraphs contribute.  Analogously to the volume $V_{g,n}(b_1,...,b_n)$ it is a non-trivial fact that $N_{g,n}(b_1,...,b_n)$ which is a sum of piecewise defined quasi-polynomials is actually a quasi-polynomial.
\begin{thm}[\cite{NorCou}]    \label{th:poly}
The number of lattice points $N_{g,n}(b_1,...,b_n)$ is a symmetric quasi-poly\-nom\-ial of degree $3g-3+n$  in the integers $(b_1^2,...,b_n^2)$ depending on the parity of the $b_i$.
\end{thm}
The dependence on the parity means that that $N_{g,n}(b_1,...,b_n)$ is polynomial on each coset of $2\bz^n \subset \bz^n$.  By symmetry, we can represent its $2^n$ polynomials by the $n$ polynomials $N^{(k)}_{g,n}(b_1,...,b_n)$, for $k=1,...,n$, symmetric in $b_1,...,b_k$ and $b_{k+1},...,b_n$ corresponding to the first $k$ variables being odd.
\[ N_{g,n}(b_1,...,b_n)=N_{g,n}^{(k)}(b_1,...,b_n),\quad {\rm first\ } k {\rm\ variables\ odd}.\]
If $k$ is odd then $N_{g,n}^{(k)}(b_1,...,b_n)=0$.
\begin{table}[ht]  \label{tab:poly}
\caption{Lattice count polynomials for even $b_i$}
\begin{spacing}{1.4}  
\begin{tabular}{||l|c|c||} 
\hline\hline
{\bf g} &{\bf n}&$N^{(0)}_{g,n}(b_1,...,b_n)$\\ \hline
0&3&1\\ \hline
1&1&$\frac{1}{48}\left(b_1^2-4\right)$\\ \hline
0&4&$\frac{1}{4}\left(b_1^2+b_2^2+b_3^2+b_4^2-4\right)$\\ \hline
1&2&$\frac{1}{384}\left(b_1^2+b_2^2-4\right)\left(b_1^2+b_2^2-8\right)$\\ \hline
2&1&$\frac{1}{2^{16}3^35}\left(b_1^2-4\right)\left(b_1^2-16\right)\left(b_1^2-36\right)\left(5b_1^2-32\right)$\\
\hline\hline
\end{tabular} 
\end{spacing}
\end{table}
For any fatgraph $\Gamma$, its incidence matrix has index ${\rm ind}_{A_{\Gamma}}=2$.  Hence
\begin{thm} \label{th:latvol}
$\displaystyle{N_{g,n}^{(k)}(b_1,...,b_n)=2V_{g,n}(b_1,...,b_n)+}$ lower order terms, ($k$ even.)
\end{thm}
An immediate corollary of Theorems~\ref{th:Kon} and \ref{th:latvol} is an identification of the top degree coefficients of $N_{g,n}(b_1,...,b_n)$ with intersection numbers on the moduli space.
\begin{cor}  \label{th:int}
For $|{\bf d}|=\sum_id_i=3g-3+n$ and ${\bf d}!=\prod d_i!$ the coefficient $c_{\bf d}$ of $b^{2{\bf d}}=\prod b_i^{2d_i}$ in $N_{g,n}^{(k)}(b_1,...,b_n)$ for $k$ even, is the intersection number
\[c_{\bf d}=\frac{1}{2^{6g-6+2n-g}{\bf d}!}\int_{\overline{\modm}_{g,n}}c_1(L_1)^{d_1}...c_1(L_n)^{d_n}.\]
\end{cor}
When all $b_i$ are even, if $P_{\Gamma}(b_1,...,b_n)$ is non-empty then it contains interior integer points and the constant term of $N_{\Gamma}(b_1,...,b_n)$ equals $\chi(P_{\Gamma})$.  Both of these can fail when some $b_i$ are odd.  For example, $N_{0,4}^{(2)}(b_1,b_2,b_3,b_4)=\frac{1}{4}\sum b_i^2-\frac{1}{2}$ and since ${\rm Aut\ }\Gamma=\{1\}$ in genus 0 then one of the $N_{\Gamma}(b_1,...,b_4)$ has constant term which is non-integral and hence not $\pm1$.  This good behaviour for even $b_i$ yields:
\begin{thm}[\cite{NorCou}]  \label{th:euler}
$N_{g,n}(0,...,0)=\chi\left(\modm_{g,n}\right)$.
\end{thm}

Kontsevich proved that the tautological intersection numbers, hence the volume polynomials $V_{g,n}(b_1,...,b_n)$, satisfy a recursion relation conjectured by Witten \cite{WitTwo} that uniquely determine them.  The lattice count quasi-polynomials $N_{g,n}(b_1,...,b_n)$ satisfy a recursion relation that uniquely determines the polynomials and when restricted to the top degree terms imply Witten's recursion. 
\begin{thm}   \label{th:recurs}
The lattice count polynomials satisfy the following recursion relation which determines the polynomials uniquely from $N_{0,3}$ and $N_{1,1}$.
{\setlength\arraycolsep{2pt} 
\begin{eqnarray}    \label{eq:rec}
\left(\sum_{i=1}^nb_i\right)N_{g,n}(b_1,...,b_n)&=&\sum_{i\neq j}\sum_{p+q=b_i+b_j}pqN_{g,n-1}(p,b_1,..,\hat{b}_i,..,\hat{b}_j,..,b_n)\nonumber\\
&+&\frac{1}{2}\sum_i\sum_{p+q+r=b_i} pqr\biggl[N_{g-1,n+1}(p,q,b_1,..,\hat{b}_i,..,b_n)\\
&&\hspace{20mm}+\hspace{-7mm}\sum_{\begin{array}{c}_{g_1+g_2=g}\\_{I\sqcup J=\{1,..,\hat{i},..,n\}}\end{array}}\hspace{-8mm}N_{g_1,|I|+1}(p,b_I)N_{g_2,|J|+1}(q,b_J)\biggr]\nonumber
\end{eqnarray}}
\end{thm}
\begin{proof}
The strategy of the proof is as follows. Construct any $\Gamma\in\fat(\bb_S)$ from smaller fatgraphs by removing from $\Gamma$ a simple subgraph $\gamma$ to get
\[ \Gamma=\Gamma'\cup\gamma.\]
The subgraph $\gamma$ is an edge or a lollipop which are the simplest subgraphs possible so that the remaining fatgraph $\Gamma'$ is legal.  Here a {\em lollipop} is a loop---a single edge with two endpoints identified---union a (possible empty) edge at a valence 3 vertex.  The length of a lollipop is the sum of the lengths of its two edges.  There are two cases for removing an edge or a lollipop from $\Gamma\in\fat(\bb_S)$, shown in Figures~\ref{fig:case1} and \ref{fig:case2}. The broken line signifies $\gamma$, and the remaining fatgraph is $\Gamma-\gamma=\Gamma'\in\fato_{g',n'}(\bb'_{S'})$ for $(g',n')=(g,n-1)$ or $(g-1,n+1)$ or $\Gamma'=\Gamma_1\sqcup\Gamma_2$ for the pair $\Gamma_i\in\fato_{g_i,n_i}(\bb_i)$, $i=1,2$ such that $g_1+g_2=g$ and $n_1+n_2=n+1$. 

In each case, the automorphism groups of $\Gamma'$ and $\Gamma$ act on the construction.  The automorphism group of $\Gamma$ sends $\gamma$ to an isomorphic copy of $\gamma$ in $\Gamma$.  The automorphism group of $\Gamma'$ acts on the locations where the ends of $\gamma$ are attached.  Both actions are transitive, or in other words the subgroup of automorphisms of $\Gamma$ that fix $\gamma$, and the subgroup of automorphisms of $\Gamma'$ that fix the endpoints $\gamma$ are trivial.  This is because if an automorphism fixes the endpoints of $\gamma$ then it fixes an adjacent oriented edge and hence is trivial.

Each fatgraph $\Gamma\in\fat(\bb_S)$ is produced in many ways, one for each edge and lollipop $\gamma\subset\Gamma$. The number of such $\gamma$ is not constant over all $\Gamma\in\fat(\bb_S)$ however a weighted count over the lengths of each $\gamma$ can be arranged to be constant as follows.  Since each half-edge of $\Gamma$ can be assigned a unique boundary component
\[|X|=\sum b_i\]
where we recall that $X$ is the set of oriented edges of $\Gamma$.
We exploit this simple fact by taking each removal of $\gamma$, an edge or lollipop, $q$ times where $\gamma$ has length $q/2$ so that we end up with $(\sum b_i)$ copies of $\Gamma$, if ${\rm Aut\ }\Gamma$ is trivial. More generally, we will explain in each case how to end up with $(\sum b_i)/|{\rm Aut\ }\Gamma|$ copies of $\Gamma$ which is a summand of $(\sum b_i)\cdot N_{g,n}(\bb_S)$, the left hand side of (\ref{eq:rec}). 

\fbox{\em Case 1} Choose a fatgraph $\Gamma'\in\fato_{g,n-1}(p, \bb_{S \setminus \{i,j\}})$ and in Case 1a add an edge of length $q/2$ inside the boundary of length $p$ so that $p+q=b_i+b_j$ as in the first diagram in Figure~\ref{fig:case1}. 

\begin{figure}[ht] 
	\centerline{\includegraphics[scale=0.3]{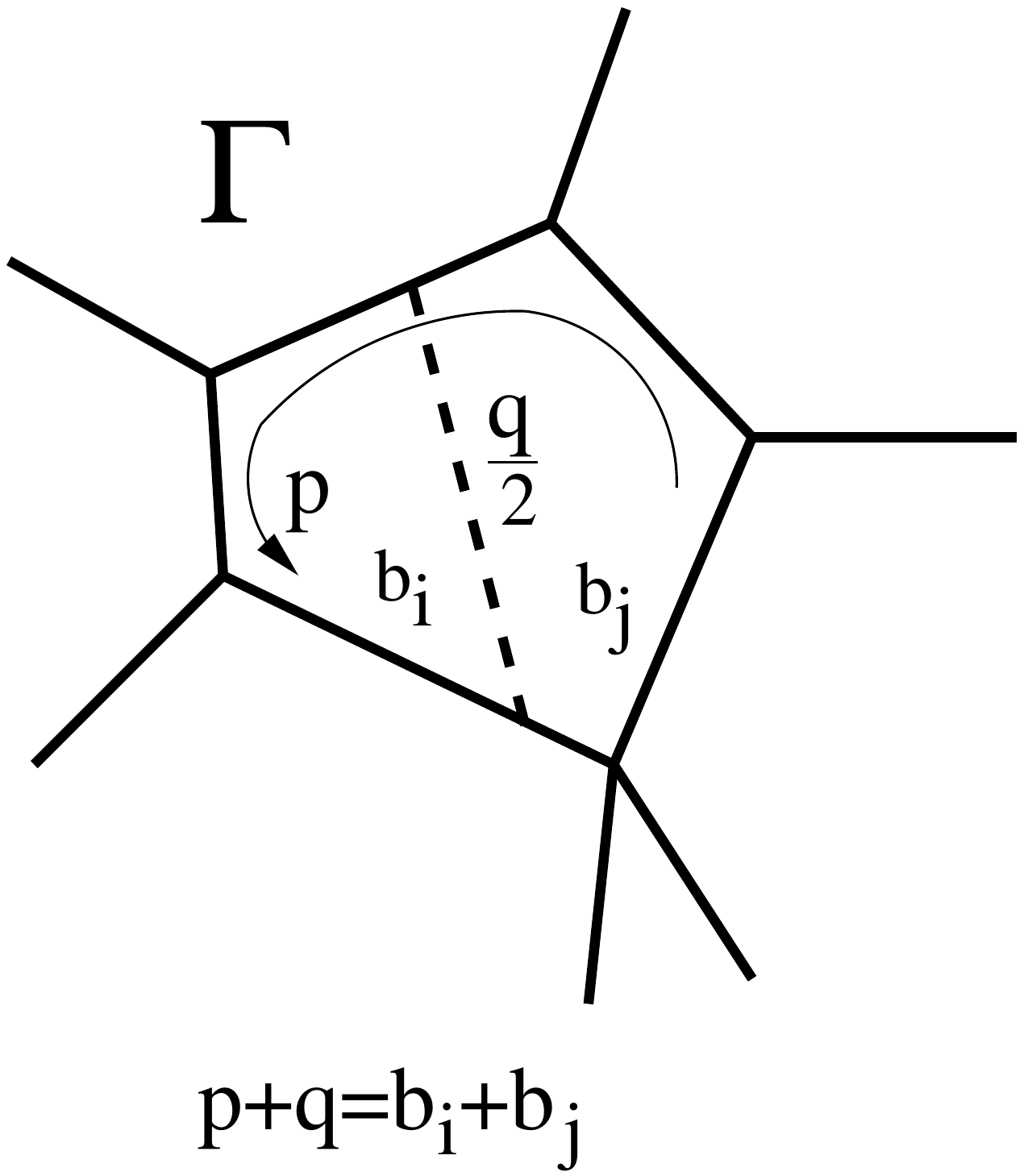} \qquad \qquad \includegraphics[scale=0.3]{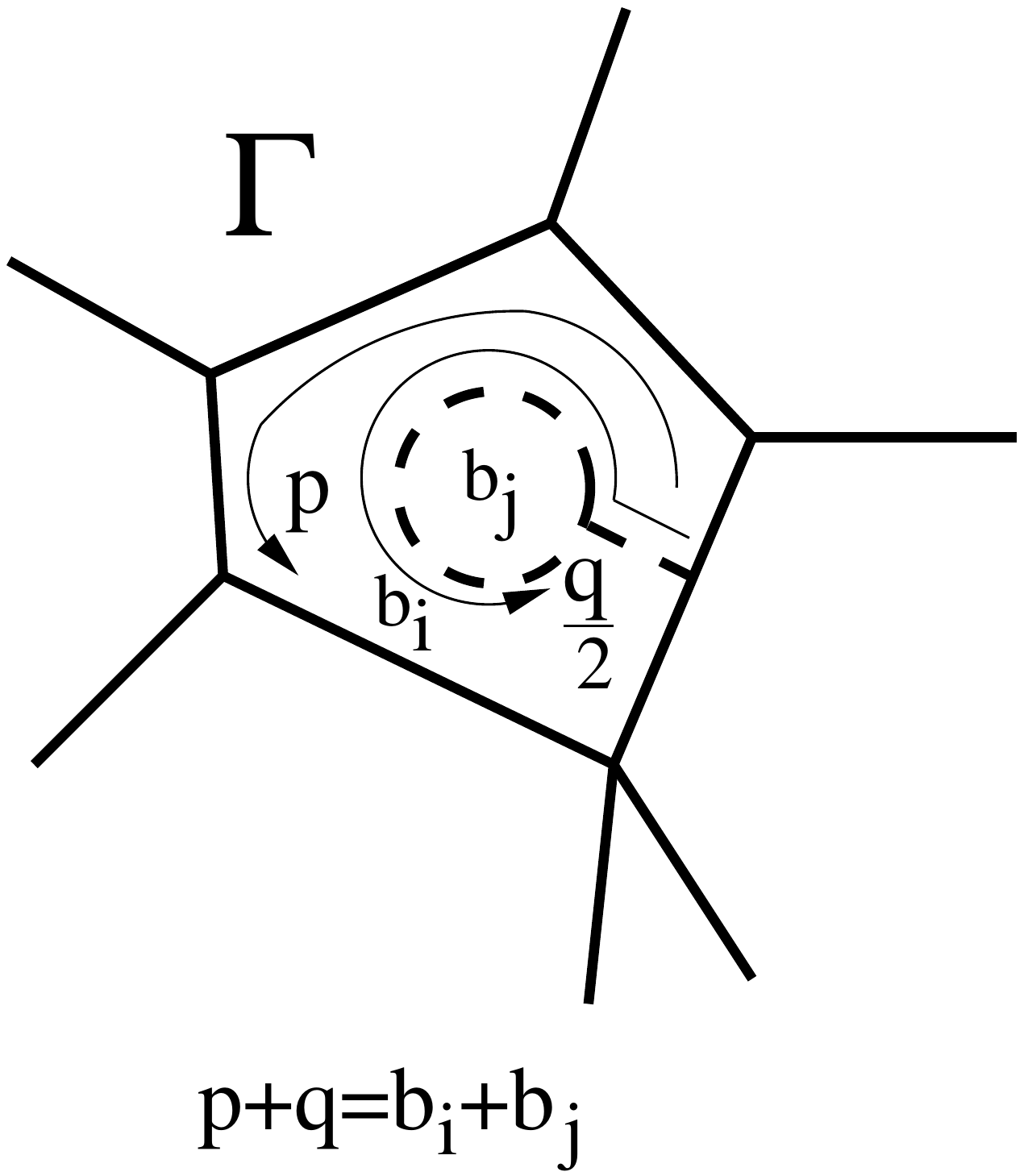}}
	\caption{a. attach edge; b. attach lollipop; to form $\Gamma$.}
	\label{fig:case1}
\end{figure}

In Case 1b attach a lollipop of total length $q/2$ inside the boundary of length $p$ as in the second diagram in Figure~\ref{fig:case1}, again so that $p+q=b_i+b_j$. In both cases for each $\Gamma'$ there are $p$ possible ways to attach the edge, and since the automorphism group of $\Gamma'$ acts transitively on the location where we attach the edge, $q$ copies of this construction produces $pq/|{\rm Aut\ }\Gamma'|$ fatgraphs. For each $\Gamma$ produced from $\Gamma'$ in this way, this construction produces $q/|{\rm Aut\ }\Gamma|$ copies of $\Gamma$.  That is, $pq/|{\rm Aut\ }\Gamma'|$ fatgraphs produce $q/|{\rm Aut\ }\Gamma|$ copies of each $\Gamma$ produced from $\Gamma'$ in this way. Applying this to all $\Gamma'\in\fato_{g,n-1}$ this construction contributes 
 \[pq N_{g,n-1}\left(p, \bb_{S \setminus \{i,j\}}\right)\] 
 to the right hand side of the recursion formula (\ref{eq:rec}) which agrees with a summand.

\fbox{\em Case 2} 
Choose a fatgraph $\Gamma'\in\fato_{g-1,n+1}(p, q, \bb_{S \setminus \{i\}})$ {\em or} $\Gamma'=\Gamma_1\sqcup\Gamma_2$ for $\Gamma_1\in\fato_{g_1,|I_1|+1}(p,\bb_{I_1})$ and $\Gamma_2\in\fato_{g_2,|I_2|+1}(q,\bb_{I_2})$ where $g_1+g_2=g$ and $I_1 \sqcup I_2 = S\setminus \{i\}$. Attach an edge of length $r/2$ connecting these two boundary components as in Figure~\ref{fig:case2} so that $p+q+r=b_i$. 
\begin{figure}[ht] 
	\centerline{\includegraphics[height=4cm]{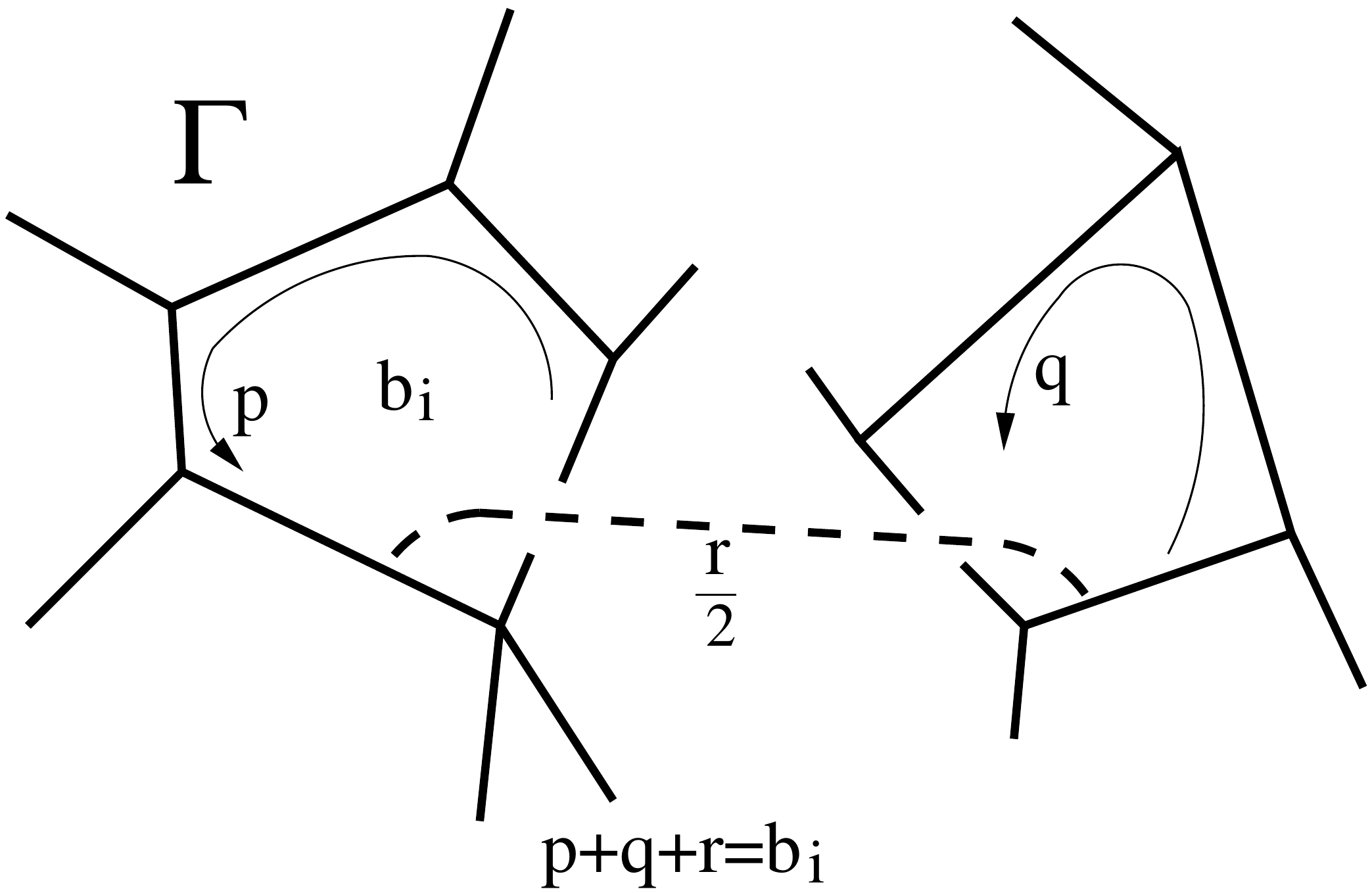}}
	\caption{$\Gamma$ is obtained from a single fatgraph or two disjoint fatgraphs by adding an edge.}
	\label{fig:case2}
\end{figure}

In the diagram, the two boundary components of lengths $p$ and $q$ are part of a fatgraph that may or may not be connected. There are $pq$ possible ways to attach the edge. An enlarged group of isomorphisms between fatgraphs $\Gamma'$ that does not necessarily preserve the labeling of the two attaching boundary components acts here because we can swap the role of the two attaching boundary components. This either identifies two different fatgraphs $\Gamma'$ or produces new automorphisms of $\Gamma'$. In the first case we count only one of them, or more conveniently we count both of them with a weight of $\frac{1}{2}$. Hence $r$ copies of this construction produces $\frac{1}{2}pqr/|{\rm Aut\ }\Gamma'|$ fatgraphs. In the second case, the action of the automorphism group of $\Gamma'$ on the locations where we attach the edges extends to an action of a larger group ${\rm Aut}^*\Gamma'$ that does not necessarily preserve the labeling of the two attaching boundary components.  Thus ${\rm Aut\ }\Gamma'$ is an index 2 subgroup of ${\rm Aut}^*\Gamma'$:
\begin{equation} \label{eq:index2}
1\to{\rm Aut\ }\Gamma'\to{\rm Aut}^*\Gamma'\to\bz_2\to1
\end{equation}
and $r$ copies of this construction produces $pqr/|{\rm Aut}^*\Gamma'|=\frac{1}{2}pqr/|{\rm Aut\ }\Gamma'|$ fatgraphs, so we again count with a weight of $\frac{1}{2}$ as above. For each $\Gamma$ produced from $\Gamma'$ in this way, this construction produces $r/|{\rm Aut\ }\Gamma|$ copies of $\Gamma$.  That is, $\frac{1}{2}pqr/|{\rm Aut\ }\Gamma'|$  fatgraphs produce $r/|{\rm Aut\ }\Gamma|$ copies of each $\Gamma$ produced from $\Gamma'$ in this way. Applying this to all $\Gamma'\in\fato_{g-1,n+1}(p, q, \bb_{S\setminus \{i\}})$ and $\Gamma'=\Gamma_1\sqcup\Gamma_2$ for all $\Gamma_1\in\fato_{g_1,j}(p, \bb_{I_1})$ and $\Gamma_2\in\fato_{g_2,n+1-j}(q, \bb_{I_2})$
this construction contributes 
\[\frac{1}{2}pqr\biggl[ N_{g-1,n+1}(p, q, \bb_{S\setminus \{i\}})
+ \hspace{-5mm}\sum_{\substack{g_1+g_2=g\\I_1 \sqcup I_2 = S\setminus \{i\}}} \hspace{-3mm} N_{g_1,|I_1|+1}(p, \bb_{I_1}) N_{g_2,|I_2|+1}(q, \bb_{I_2})\biggr]\] 
 to the right hand side of the recursion formula (\ref{eq:rec}) which agrees with a summand.

By removing any edge or lollipop from $\Gamma\in\fat(\bb_S)$ we see that it can be produced (many times) using the two constructions above. Each construction produces $\Gamma$ weighted by the factor $2|\gamma|/|{\rm Aut\ }\Gamma|$ where $|\gamma|$ is the length of the edge or lollipop. The sum over $|\gamma|$ for all edges or lollipops $\gamma\subset\Gamma$ yields the number of edges of $\Gamma$ so using $|X|=\sum b_i$ this gives a weight of $(\sum b_i)/|{\rm Aut\ }\Gamma|$ to each $\Gamma\in\fat(\bb_S)$. The weighted sum over all $\Gamma\in\fat(\bb_S)$ is thus $(\sum b_i)N_{g,n}(\bb_S)$ which gives the left hand side of (\ref{eq:rec}) and completes the proof.

\end{proof}

To apply the recursion we need to first calculate $N_{0,3}(b_1,b_2,b_3)$ and $N_{1,1}(b_1)$.  There are seven labeled fatgraphs in $\fato_{0,3}$ coming from three unlabeled fatgraphs.  It is easy to see that $N_{0,3}(b_1,b_2,b_3)=1$ if $b_1+b_2+b_3$ is even (and 0 otherwise.)  This is because for each $(b_1,b_2,b_3)$ there is exactly one of the seven labeled fatgraphs $\Gamma$ with a unique solution of $A_{\Gamma}{\bf x}={\bf b}$ while the other six labeled fatgraphs yield no solutions.  For example, if $b_1>b_2+b_3$ then only the dumbbell fatgraph $\Gamma$ with 
\[A_{\Gamma}=\left(\begin{array}{ccc}2&1&1\\0&1&0\\0&0&1\end{array}\right)\] 
has a solution and that solution is unique.

To calculate $N_{1,1}(b_1)$, note that $A_{\Gamma}=[2\quad 2\quad 2]$ or $[2\quad 2]$ for the 2-vertex and 1-vertex fatgraphs.  Hence
\[ N_{1,1}(b_1)=a_1\binom{\frac{b_1}{2}-1}{2}+a_2\binom{\frac{b_1}{2}-1}{1}\]
where $a_1$ is the number of trivalent fatgraphs (weighted by automorphisms) and $a_2$ is the number of 1-vertex fatgraphs.  The genus 1 graph $\Gamma$ from Figure~\ref{fig:fat} has $|Aut\Gamma|=6$ so $a_1=1/6$, and $a_2$ uses the genus 1 figure 8 fatgraph which has automorphism group $\bz_4$ hence $a_2=1/4$.  Thus
\[ N_{1,1}(b_1)=\frac{1}{6}\binom{\frac{b_1}{2}-1}{2}+\frac{1}{4}\binom{\frac{b_1}{2}-1}{1}=\frac{1}{48}\left(b_1^2-4\right).\]
We can also calculate $N_{1,1}(b_1)$ via edge removal
\[ b_1N_{1,1}(b_1)=\frac{1}{2}\sum_{\begin{array}{c}p+q+p=b\\b{\rm\ even}\end{array}}pq.\]
We will calculate $N_{0,4}[b_1,b_2,b_3,b_4]$ to demonstrate the recursion relation and the parity issue.
\[\left(\sum_{i=1}^4b_i\right)N_{0,4}(b_1,b_2,b_3,b_4)=\sum_{i\neq j}\sum_{\begin{array}{c}\scriptstyle p+q=b_i+b_j\\\scriptstyle q{\rm\ even}\end{array}}pq.\]
If all $b_i$ are even, or all $b_i$ are odd, then $b_i+b_j$ is always even so the sum is over $p$ and $q$ even.  We have
\[ \sum_{i\neq j}\sum_{\begin{array}{c}\scriptstyle p+q=k\\\scriptstyle q{\rm\ even}\end{array}}pq=4\binom{\frac{k}{2}+1}{3}\]
so
\[\left(\sum_{i=1}^4b_i\right)N^{(0)}_{0,4}({\bf b})=\sum_{i\neq j}4\binom{\frac{b_i+b_j}{2}+1}{3}=\left(\sum_{i=1}^4b_i\right)\frac{1}{4}\left(b_1^2+b_2^2+b_3^2+b_4^2-4\right)\]
agreeing with Table~\ref{tab:poly}.  If $b_1$ and $b_2$ are odd and $b_3$ and $b_4$ are even then we need
\[ \sum_{i\neq j}\sum_{\begin{array}{c}\scriptstyle p+q=k\\\scriptstyle q{\rm\ even}\end{array}}pq=\frac{1}{2}\binom{k+1}{3}\]
so
\begin{eqnarray*}
\left(\sum_{i=1}^4b_i\right)N^{(2)}_{0,4}({\bf b})&=&
\hspace{-8mm}\sum_{(i,j)=(1,2){\rm\ or\ }(3,4)}\hspace{-3mm}4\binom{\frac{b_i+b_j}{2}+1}{3}
+\hspace{-3mm}\sum_{(i,j)\neq(1,2){\rm\ or\ }(3,4)}\hspace{-1mm}\frac{1}{2}\binom{b_i+b_j+1}{3}\\
&=&\left(\sum_{i=1}^4b_i\right)\frac{1}{4}\left(b_1^2+b_2^2+b_3^2+b_4^2-2\right)
\end{eqnarray*}
so we see that the polynomial representatives of $N_{0,4}({\bf b})$ agree up to a constant term.\\

The proof of the polynomial behaviour of $N_{g,n}(b_1,...,b_n)$ is different to the proof of the polynomial behaviour of $V_{g,n}(b_1,...,b_n)$.  Kontsevich proved that $V_{g,n}(b_1,...,b_n)$ is a polynomial by identifying its coefficients with intersection numbers over the moduli space.  Whereas, $N_{g,n}(b_1,...,b_n)$ is proven to be quasi-polynomial using the recursion relation given in Theorem~\ref{th:recurs}.  It would be good to have alternative proofs of both of these facts, perhaps via the local behaviour of fatgraph diagrams.  An interpretation of the general coefficients of the quasi-polynomial $N_{g,n}(b_1,...,b_n)$ is not yet known.  In \cite{DNoCou} the ideas here have been extended to define and count lattice points in the moduli space $\overline{\modm}_{g,n}$ of stable genus $g$ curves with $n$ labeled points.  This may lead to an interpretation of the general coefficients of $N_{g,n}(b_1,...,b_n)$.

It would also be interesting to understand a deeper explanation for the fact that $N_{g,n}(b_1,...,b_n)$ is a quasi-polynomial in the squares $b_i^2$.   The volume polynomial $V_{g,n}(b_1,...,b_n)$ is a polynomial in the squares $b_i^2$ because its coefficients are Chern numbers on the moduli space which only appear in even degree.  This also reflects the fact that the moduli space can be constructed in the algebraic category.  The degree $6g-6+2n$ terms of $N_{g,n}(b_1,...,b_n)$ are polynomial in the $b_i^2$ because they coincide with $V_{g,n}(b_1,...,b_n)$.  The degree $6g-7+2n$ terms of $N_{g,n}(b_1,...,b_n)$ are related to volumes of codimension 1 faces.  The cancelation of these terms resembles the behaviour of the fundamental class of an orientable compact manifold.  

\subsection{The $n=1$ case}

The case $n=1$ corresponds to the moduli space of 1-pointed curves $\modm_{g,1}$ and to branched covers of $S^2$ with cyclic ramification over $\infty$ (see the next section.)  The lattice point count quasi-polynomial $N_{g,1}(b)$ is simpler to interpret than in the general case of $n>1$.   This is because the incidence matrix is so simple that the number of lattice points can be calculated explicitly.  For any $\Gamma\in\fato_{g,1}$ the incidence matrix is $A_{\Gamma}=[2,2,...,2]$.  The equation $Ax=b$ has $\binom{\frac{b}{2}-1}{e(\Gamma)-1}$ positive integral solutions when $b$ is even (and no solutions when $b$ is odd.)  Hence
\[N_{g,1}(b)=c_{6g-3}^{(g)}\binom{\frac{b}{2}-1}{6g-4}+c_{6g-4}^{(g)}\binom{\frac{b}{2}-1}{6g-5}+..+c_k^{(g)}\binom{\frac{b}{2}-1}{k-1}+..+c_{2g}^{(g)}\binom{\frac{b}{2}-1}{2g-1}\]
where the coefficients are weighted counts of fatgraphs of genus $g$ with one boundary component
\[ c_k^{(g)}=\sum_{\begin{array}{c}\Gamma\in\fato_{g,1}\\e(\Gamma)=k\end{array}}\frac{1}{|Aut\Gamma|}.\]
The polynomial $\binom{\frac{b}{2}-1}{k}$ evaluates at $b=0$ to give $(-1)^k$ hence
\[N_{g,1}(0)=\sum_{\Gamma\in\fato_{g,1}}\frac{(-1)^{e(\Gamma)-1}}{|Aut\Gamma|}=\chi(\modm_{g,1})\]
directly showing that the Euler characteristic of the moduli space is given by evaluation of the polynomial at 0.  This agrees with the more general result of Theorem~\ref{th:euler}. 

The polynomial $N_{g,1}(b)$ appears implicitly in the work of Harer and Zagier \cite{HZaEul} where they
 first calculated the orbifold Euler characteristic of $\modm_{g,1}$
\[\chi\left(\modm_{g,1}\right)=\zeta(1-2g).\] 
This is used together with the exact sequence of mapping class groups
\begin{equation}  \label{eq:exact} 
1\rightarrow\pi_1(C-\{p_1,...,p_n\})\rightarrow\Gamma_g^{n+1}\rightarrow\Gamma_g^n\rightarrow 1
\end{equation}
to calculate $\chi(\modm_{g,n+1})=\chi(\Gamma_g^{n+1})=\chi(\Gamma_g^n)\chi(C-\{p_1,...,p_n\})$ hence
\begin{align*}
&\chi(\modm_{g,n+1})=(-1)^n\frac{(2g-2+n)!}{(2g-2)!}\chi(\modm_{g,1}),\quad g>0\\ 
&\chi(\modm_{0,n+1})=(-1)^n(n-2)!\chi(\modm_{0,3}).
\end{align*}
\begin{thm}[\cite{NorStr}]   \label{th:dilaton}
\[N_{g,n+1}(2,b_1,...,b_n)-N_{g,n+1}(0,b_1,...,b_n)=(2g-2+n)N_{g,n}(b_1,...,b_n).\]
\end{thm}
Using Theorem~\ref{th:euler} together with the vanishing result $N_{g,n+1}(2,0,...,0)=0$, \cite{NorStr},  Theorem~\ref{th:dilaton} is seen to generalise the recursion between Euler characteristics of moduli spaces.  In a sense it reflects the exact sequence (\ref{eq:exact}).

The following combinatorial definitions appear in \cite{HZaEul}.
\begin{defn}
Define $\mu_g(n)$ to be the number of orientable genus $g$ gluings of a $2n$-gon with a distinguished edge and so that neighbouring edges are not identified.
\end{defn}
A fatgraph $\Gamma$ with one boundary component of length $b$ defines a $b$-gon with edges identified such that neighbouring edges are not identified since otherwise this would correspond to a valence one vertex in the fatgraph which is excluded.  Furthermore,  the automorphism group of $\Gamma$ is a subgroup of the cyclic group ${\rm Aut\ }\Gamma\subset\bz_b$ so we must divide by
\[ b=\{ {\rm number\ of\ ways\ to\ distinguish\ an\ edge\ of\ }\Gamma\}\times|{\rm Aut\ }\Gamma|\]  
to get the relation
\[ N_{g,1}(b)=\frac{1}{b}\mu_g\left(\frac{b}{2}\right).\]

The polynomials $N_{g,1}(b)$ are calculated using the recursion (\ref{eq:rec}).  This requires one to first calculate $N_{g',n'}$ for all $(g',n')$ satisfying $2g'-2+n'<2g-1$.  Instead, the techniques of Harer and Zagier allow one to calculate $N_{g,1}(b)$ without calculating $N_{g',n'}(b_1,...,b_{n'})$ for $n'>1$.
\begin{defn}
Define $\epsilon_g(n)$ to be the number of orientable genus $g$ gluings of a $2n$-gon with a distinguished edge.  
\end{defn}
\begin{lemma}[\cite{HZaEul}]
\[\epsilon_g(n)=\sum_{i\geq 0}\binom{2n}{i}\mu_g(n-i)\]
\end{lemma}
Store the $\epsilon_g(n)$ in the generating function
\[C(n.k)=\sum_{0\leq g\leq n/2}\epsilon_g(n)k^{n+1-2g}.\]
\begin{thm}[\cite{HZaEul}]
$C(n,k)=(2n-1)!!c(n,k)$ where $c(n,k)$ is defined by the generating function
\[ 1+2\sum_{n=0}^{\infty}c(n,k)x^{n+1}=\left(\frac{1+x}{1-x}\right)^k\]
or by the recursion
\[c(n,k)=c(n,k-1)+c(n-1,k)+c(n-1,k-1)\quad n,k>0\]
with boundary conditions $c(0,k)=k$, $c(n,0)=0$, $n,k\geq 0$.
\end{thm}
We have used this to calculate $N_{g,1}(b)$ up to $g=50$.  An interesting question is whether the large genus behaviour of $N_{g,1}(b)$ can be understood.  For example, does $N_{g,1}(b)/N_{g,1}(0)$ tend to a known non-trivial function as $g\to\infty$?

\section{Hurwitz problems}   \label{sec:hur}

A Hurwitz problem is an enumerative problem in combinatorics and geometry where one counts up to isomorphism branched covers over a surface $\Sigma\to\Sigma'$ with specified branch points and ramification above the branch points.  Variations of the problem may also further require that the covers be connected, or that the points above a branch point be labeled.   Two covers $\Sigma_1$ and $\Sigma_2$ are isomorphic if there exists a homeomorphism $f:\Sigma_1\to\Sigma_2$ (that  preserves any labeling upstairs and) covers the identity downstairs on $\Sigma'$.  If $\Sigma_1=\Sigma_2$ a homeomorphism is a non-trivial automorphism of the cover if it is not isotopic to the identity.  For a degree $d$ cover the ramification above a branch point is described by a partition $\lambda=(\lambda_1,...,\lambda_k)$ of $d$, where $\lambda_i$ are integers, $\lambda_1\geq\lambda_2\geq...\geq\lambda_k>0$ and $\lambda_1+...+\lambda_k=d$.  Fix $r+1$ points $\{p_1,p_2,...,p_r,p_{r+1}\}\subset\Sigma'$ together with a partition of $d$ at each $p_i$.  The {\em Hurwitz number} is defined to be the weighted number of non-isomorphic branched covers $\pi:\Sigma\to\Sigma'$ with this branching data where the weight of a cover $\pi$ is $1/|{\rm Aut\ }\pi|$.   Hurwitz numbers for disconnected and connected covers can be retrieved from each other, and the Hurwitz number of labeled covers is simply $|{\rm Aut\ }\mu|$
times the unlabeled Hurwitz number where, say the points $\pi^{-1}(p_{r+1})$ are labeled and $\mu$ is the partition at $p_{r+1}$.  
\begin{ex}   \label{ex:hur}
Over $\{0,1,\infty\}\subset\bp^1$ specify branching $\{(4),(2,2),(4)\}$, respectively.  The cover $\Sigma$ is a genus 1 surface since by the Riemann-Hurwitz formula 
\[\chi(\Sigma)=4\cdot 2-\left((4-1)+(2-1)+(2-1)+(4-1)\right)=0.\]  
There is a unique cover $\Sigma\to S^2$ with this branching data.  It possesses a non-trivial homeomorphism $f:\Sigma\to\Sigma$ that swaps the two points above $1$ and cyclically permutes the points above $0$ and infinity.  In fact $f$ generates the automorphism group of this cover which is $\bz_4$.  Thus the number of branched covers is $1/4$.
\end{ex}
An equivalent formulation of a Hurwitz problem consists of counting factorisations in the symmetric group.  A partition $\lambda=(\lambda_1,...,\lambda_k)$ of $d$ determines a conjugacy class in the symmetric group $C_{\lambda}\subset S_d$ consisting of all permutations with cycle structure $(\lambda_1,...,\lambda_k)$.  It is also natural to represent a conjugacy class as the formal sum of its elements which is an element of the group ring $C_{\lambda}\in \bc(S_d)$, in fact it is an element of centre of the group ring $C_{\lambda}\in Z(\bc(S_d))$.  Given conjugacy classes $C_1$, $C_2$, ..., $C_{r+1}$ in $S_d$ the Hurwitz problem counts factorisations of the identity by elements taken from the given conjugacy classes, up to equivalence.
\[H_{C_1,...,C_{r+1}}=\#\{\sigma_1,...,\sigma_{r+1}|\prod\sigma_i=(1),\quad \sigma_i\in C_i\}/d!\]
where division by $d!$ comes from identifying equivalent products, or explicitly 
\[(\sigma_1,...,\sigma_{r+1})\sim(g\sigma_1g^{-1},...,g\sigma_{r+1}g^{-1})\] 
for any $g\in S_d$.  If a product is fixed by conjugation then this defines an automorphism.  Equivalently,
consider the representation of $\bc(S_d)$ onto itself and identify any element of $\bc(S_d)$ with the endomorphism it defines.  This is a dimension $d!$ representation hence $\tr (1)=d!$.  Any non-trivial permutation $\sigma\in S_d$ has all diagonal entries zero so in particular $\tr(\sigma)=0$.  Thus
\[H_{C_1,...,C_2}=\frac{1}{d!^2}\tr(C_1C_2...C_{r+1})\]
since the trace picks out the constant term in the expansion of $C_1C_2...C_{r+1}$.  Explicitly, the product of conjugacy classes contains all products of elements $\sigma_1...\sigma_{r+1}$ for $\sigma_i\in C_i$ and 
\[ \tr(\sigma_1...\sigma_{r+1})=0\Leftrightarrow \sigma_1...\sigma_{r+1}\neq(1)\]
so the trace picks out the terms $\sigma_1...\sigma_{r+1}=(1)$ and evaluates them to $d!$.  Division by $d!$ counts factorisations, and division by $d!^2$ counts factorisations up to equivalence.

\addtocounter{ex}{-1}
\begin{ex}[{\bf continued}] Let $C_1$, $C_2$ and $C_3$ be conjugacy classes in $S_4$ given by the respective partitions $(4)$, $(2,2)$ and $(4)$.  There are 3 possible words of type $(2,2)$, for example $(12)(34)$.  The products $(1324)(12)(34)(1324)=(1)$ and $(1423)(12)(34)(1423)=(1)$ are the only solutions to the factorisation problem involving $(12)(34)$.  Thus there are 6 solutions of the factorisation problem, 2 for each word of type $(2,2)$.  Thus the count is $6/4!=1/4$ which agrees with the calculations above.
\end{ex}

\subsection{Simple Hurwitz numbers}
Hurwitz \cite{HurHur} studied the problem of connected surfaces $\Sigma$ of genus $g$ covering $\Sigma'=S^2$, branched over $r+1$ fixed points $\{p_1,p_2,...,p_r,p_{r+1}\}$ with arbitrary partition $\mu=(\mu_1,...,\mu_n)$ over $p_{r+1}$.  Over the other $r$ branch points one specifies simple ramification, i.e. the partition $(2,1,1,....)$.   The Riemann-Hurwitz formula determines the number $r$ of simple branch points via $2-2g-n=|\mu|-r$.  
\begin{defn}  \label{def:hurn}
Define the simple Hurwitz number $H_{g,\mu}$ to be the weighted count of genus $g$ connected covers of $S^2$ with ramification $\mu=(\mu_1,...,\mu_n)$ over $\infty$ and simple ramification elsewhere.
Each cover $\pi$ is counted with weight $1/|{\rm Aut}(\pi)|$.
\end{defn}

Simple Hurwitz numbers naturally arise via the intersection of subvarieties in the moduli space $\modm_g(\bp^1,d)$ of morphisms from smooth genus $g$ curves to $\bp^1$.  If $\mu=1^d$ then by allowing the images of the simple branch points $p_1,...,p_r$ to vary, this gives an open dense subset of $\modm_g(\bp^1,d)$ since generically an element of $\modm_g(\bp^1,d)$ has simple branching.  The points $p_1,...,p_r$ give local parametrisations of this open dense subset of $\modm_g(\bp^1,d)$.  Fixing a branch point $p_j$ corresponds to intersecting with a hypersurface in $\modm_g(\bp^1,d)$.  Hence fixing $r$ branch points expresses the simple Hurwitz numbers as the intersection of $r$ hypersurfaces in $\modm_g(\bp^1,d)$.  For general $\mu$, one fixes $p_{r+1}$ and allows $p_1,...,p_r$ to vary to get an open dense subset of a subset $\modm_g(\bp^1,d,\mu)\subset\modm_g(\bp^1,d)$ which specifies the monodromy $\mu$ around $p_{r+1}$.

The ELSV formula \cite{ELSVHur} relates the Hurwitz numbers $H_{g,\mu}$ to intersection numbers of tautological classes $\psi_i=c_i(L_i)$ and $\lambda_i=c_i(E)$ where $L_i$ are tautological line bundles and $E$ is the Hodge bundle over the compactified moduli space.  The ELSV formula is
\[ H_{g,\mu}=\frac{r(g,\mu)!}{|{\rm Aut\ }\mu|}\prod_{i=1}^n\frac{\mu_i^{\mu_i}}{\mu_i!}P_{g,n}(\mu_1,...,\mu_n)\]
for the polynomial
\[ P_{g,n}(\mu_1,...,\mu_n)=\int_{\overline{\modm}_{g,n}}\frac{1-\lambda_1+...+(-1)^g\lambda_g}{(1-\mu_1\psi_1)...(1-\mu_n\psi_n)}\]
where $\mu=(\mu_1,...,\mu_n)$ and $r(g,\mu)=2g-2+n+|\mu|$.

\begin{table}[ht]  \label{tab:pol}
\caption{$P_{g,n}(\mu_1,...,\mu_n)$}
\begin{spacing}{1.4}  
\begin{tabular}{||l|c|c||} 
\hline\hline

{\bf g} &{\bf n}&$P_{g,n}(\mu_1,...,\mu_n)$\\ \hline

0&3&1\\ \hline
1&1&$\frac{1}{24}(\mu_1-1)$\\ \hline
0&4&$\mu_1+\mu_2+\mu_3+\mu_4$\\ \hline
1&2&$\frac{1}{24}\left(\mu_1^2+\mu_1\mu_2+\mu_2^2-\mu_1-\mu_2\right)$\\ 
\hline\hline
\end{tabular} 
\end{spacing}
\end{table}
The top degree terms of $P_{g,n}(\mu_1,...,\mu_n)$ involve only intersection numbers of $\psi$-classes and hence are related to Kontsevich's volume polynomial $V_{g,n}(b_1,...,b_n)$.  Later we will need the following definition.
\begin{defn}
Define the labeled simple Hurwitz number
\[H_{g,n}(\mu_1,...,\mu_n)=\frac{|{\rm Aut\ }\mu|}{r(g,\mu)!}\cdot H_{g,\mu}.\]
\end{defn}
The points above $\infty$ are labeled, contributing a factor of $|{\rm Aut\ }\mu|$, and now $(\mu_1,...,\mu_n)\in\bz_+^n$ (not necessarily monotonic.)  The factor $1/r(g,\mu)!$ is there mainly because this is a convenient form for Section~\ref{sec:lap}.  

\subsection{A Belyi Hurwitz problem}   \label{sec:belyi}
Consider connected orientable genus $g$ branched covers $\pi:\Sigma\to S^2$ unramified over $S^2-\{0,1,\infty\}$ with points in the fibre over $\infty$ labeled $(p_1,...,p_n)$ and with ramification $(b_1,...,b_n)$, ramification $(2,2,...,2)$ over $1$ and ramification greater than 1 at all points above $0$.  We call the weighted count of non-isomorphic such branched covers a Belyi Hurwitz problem because the covers are known as Belyi maps.  Such a branched cover $\pi:\Sigma\to S^2$ decomposes $\Sigma$ into flat cylinders as follows.  Pull back the interval $\pi^{-1}[0,1]=\Gamma\subset\Sigma$ to define a fatgraph $\Gamma$ with integer side lengths.  It is a fatgraph since $\Sigma-\Gamma\to\overline{\bc} -[0,1]$ is branched only over $\infty$ hence each component is a cyclic branched cover of the disk at a single point, which is necessarily a disk, i.e. $\Sigma-\Gamma$ is a disjoint union of disks.  
\begin{figure}[ht] 
	\centerline{\includegraphics[height=2cm]{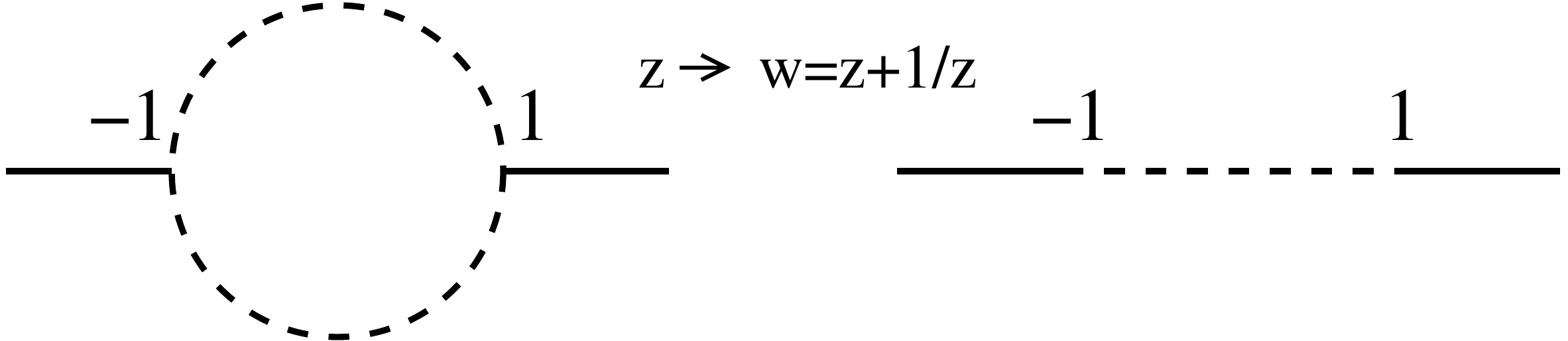}}
	\caption{cylinder $\simeq\bc-[0,1]$ }
	\label{fig:confeq}
\end{figure}
Figure~\ref{fig:confeq} shows that the complement of an interval in the complex plane is conformally a flat cylinder 
\[S^1\times\br^+\simeq\bc-[0,1].\]  
More precisely, $S^1\times\br^+$ is isometric to $\{z\in\bc:|z|>1\}$ equipped with the metric $|dz/2\pi z|^2$.  It maps conformally to $\{w\in\bc:w\notin[-1,1]\}$ via $w=z+1/z$.  Translate and rescale to adjust the interval $[-1,1]$ to $[0,1]$.   

Over the flat cylinder $\bc-[0,1]$, $\pi$ is a covering map, and in particular a degree $b_i$ cover of $\bc-[0,1]$ is a flat cylinder of circumference $b_i$.  Thus, if we prescribe ramification $(b_1,...,b_n)$ over $\infty$, $(2,2,...,2)$ over 1 and $(r_1,...,r_k)$ over 0, where $r_i>1$, then $\pi^{-1}(\bc-[0,1])$ decomposes the surface upstairs into flat cylinders and $\pi^{-1}[0,1]$ defines a lattice point in $P_{\Gamma}(b_1,...,b_n)$ for $\Gamma\in\fat$ . 

Conversely, a lattice point in the combinatorial moduli space $\modm_{g,n}(b_1,...,b_n)$ corresponds to a Riemann surface $\Sigma$ with labeled points $\{p_1,...,p_n\}$ whose complement decomposes along flat cylinders with integer circumferences $(b_1,...,b_n)$.   Any two flat cylinders that meet, meet along integer edge length arcs.  The problem of gluing together cylinders along arcs on their boundaries is solved elegantly by using the conformal model of the cylinder $S^1\times\br^+\simeq\bc-[0,1]$ described above.  The conformal equivalence can extend across $[0,1]$ enabling one to glue cylinders along arcs on their boundaries.  This gives rise to a branched cover $\pi:\Sigma\to S^2$ unramified over $S^2-\{0,1,\infty\}$ with points in the fibre over $\infty$ labeled $(p_1,...,p_n)$ and with ramification $(b_1,...,b_n)$, ramification $(2,2,...,2)$ over $1$ and ramification greater than 1 at all points above $0$.  

Hence we have proven that $N_{g,n}(b_1,...,b_n)$ is a sum of Hurwitz numbers.
\begin{prop}
$N_{g,n}(b_1,...,b_n)$ is the weighted number of non-isomorphic connected branched covers of $S^2$ branched over $\infty$, $1$ and $0$ with labeled points of ramification $(b_1,...,b_n)$ over $\infty$, $(2,2,...,2)$ over 1 and $(r_1,...,r_k)$ over 0 with $r_i>1$.  Sum over all permitted $(r_1,...,r_k)$.
\end{prop}
If one sums over all possible ramifications $(r_1,...,r_k)$ i.e allow some $r_i=1$, then the weighted count $M_{g,n}(b_1,...,b_n)$ is no longer polynomial as it includes factorial terms.  Still, $M_{g,n}(b_1,...,b_n)$ is determined by and determines $N_{g,n}(b_1,...,b_n)$.

It is worth repeating the calculation from Exercise~\ref{ex:hur} one more time using the fatgraph perspective.  The preimage of the interval $[0,1]$ defines a fatgraph in the surface $\Sigma$.  We need to count all such fatgraphs.
\addtocounter{ex}{-1}
\begin{ex}[{\bf continued}]
Given a degree 4 branched cover $p:\Sigma\to S^2$, with ramification $\{(4),(2,2),(4)\}$ over $\{0,1,\infty\}\subset\bp^1$, the preimage of the interval $[0,1]$ defines a fatgraph in the genus 1 surface $\Sigma$.  The partition $(4)$ shows that there is a 4-valent vertex of the fatgraph above 0.  The vertices of the fatgraph above 1 are 2-valent, or in other words they denote the centres of edges.  Thus the fatgraph has 2 edges and one 4-valent vertex, hence it is the figure 8. 
\begin{figure}[ht]  
	\centerline{\includegraphics[height=3cm]{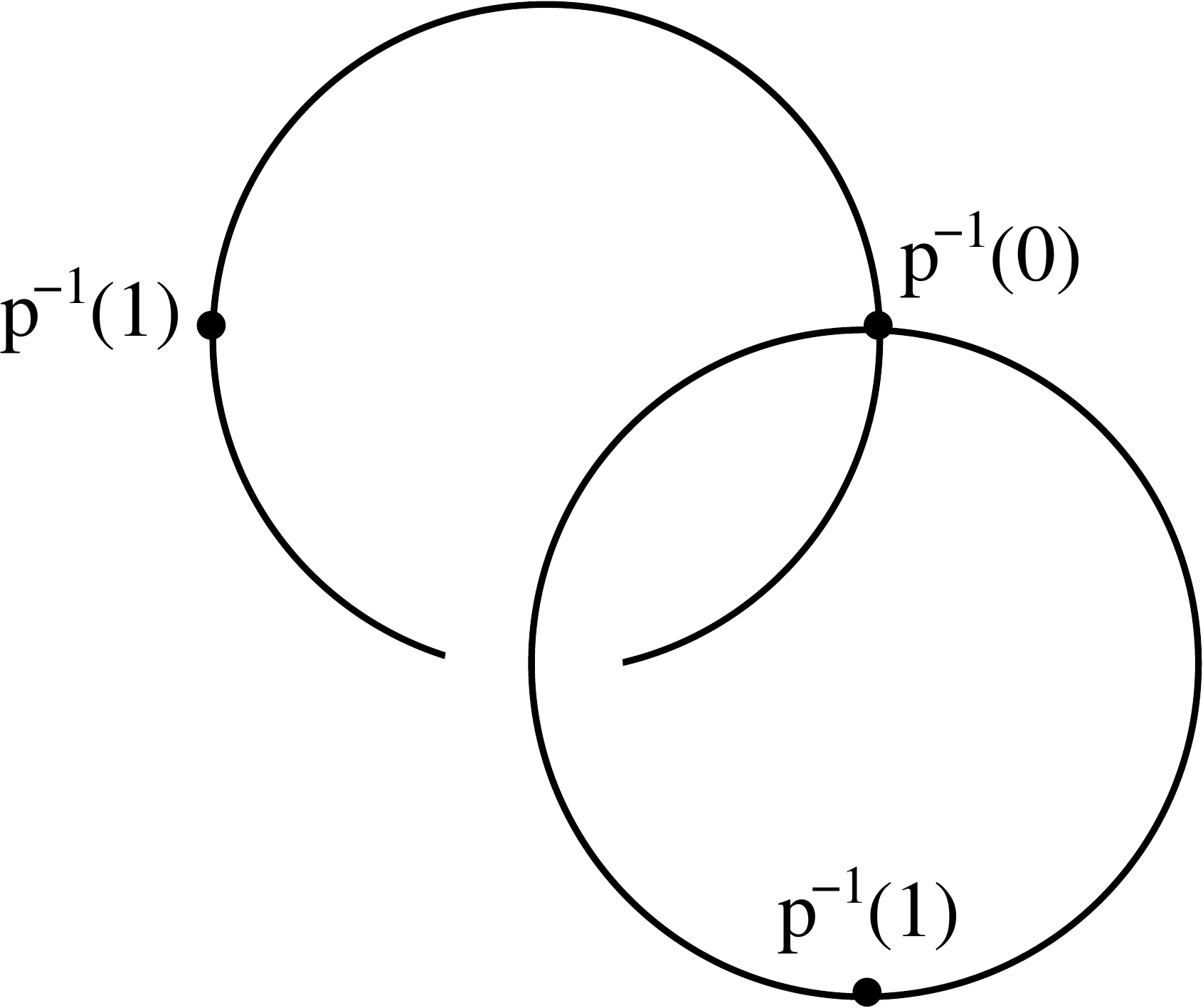}}
	\caption{Preimage of $[0,1]$ is a fatgraph}
	\label{fig:fig8}
\end{figure}
It has automorphism group $\bz_4$ hence the weighted number of fatgraphs is $1/4$.
\end{ex}
Each integral point in the polytope $P_{\Gamma}(b_1,...,b_n)$ is an integer length fatgraph which is also known as a clean {\em dessin d'enfant} defined by Grothendieck \cite{GroEsq}.  The corresponding branched covering $p:\Sigma\to S^2$ is known as a Belyi map.  The adjective ``clean" refers to the ramification $(2,2,...,2)$ of $p$ over 1.  A dessin d'enfant represents a curve in $\modm_{g,n}$ defined over $\overline{\bq}$ so we see that $N_{g,n}(b_1,...,b_n)$ counts only curves defined over $\overline{\bq}$.    See \cite{LZvGra,MPeRib}.
\begin{ex}  \label{ex:dessin}
The polynomial $N_{0,4}(b,b,b,b)=b^2-1$ counts curves defined over $\overline{\bq}$.  A point in $\modm_{0,4}$ corresponding to a curve defined over $\overline{\bq}$ is represented by an algebraic number in $\bc-\{0,1\}$ via the isomorphism $\modm_{0,4}\to\bc-\{0,1\}$ that sends four points in $S^2$ to their cross-ratio.  When $b$ is small, the curves are rather symmetric and the Belyi maps are close to Galois.  

In his famous book \cite{KleVor} of 1884 on Galois theory, Felix Klein described the ring of (projectively) invariant forms for the platonic groups.  We will need polynomials for the tetrahedral symmetry group $A_4$ and the octahedral symmetry group $S_4$ with generators
\[ A_4=\langle z\mapsto -z, 1/z, (z+i)/(z-i)\rangle\subset S_4=\langle z\mapsto iz, 1/z, (z+i)/(z-i)\rangle\]
\begin{itemize}
\item for $A_4$, $K_{\pm}=z^4\pm2\sqrt{3}z^2+1$ 
\item for $S_4$, $K_e=z^5-z$,\quad $K_f=z^8+14z^4+1$.
\end{itemize}
The zeros of $K_{\pm}$ correspond to vertices and faces of the tetrahedron, the zeros of $K_e$ (and $\infty$) correspond to edges of the tetrahedron and vertices of the octahedron, and the zeros of $K_f$ correspond to faces of the octahedron.

\noindent \fbox{$b=2$}  

$N_{0,4}(2,2,2,2)=3$ determines the unique fatgraph in figure~\ref{fig:fat04a} with edge lengths 1, hence boundary lengths 2, and all possible labelings.
\begin{figure}[ht]  
	\centerline{\includegraphics[height=2.5cm]{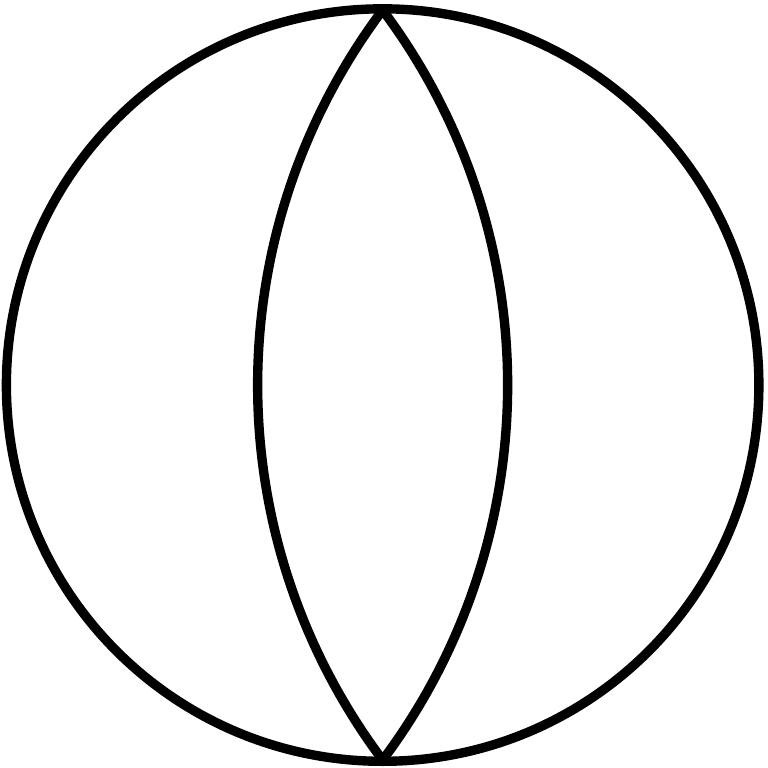}}
	\caption{Branched cover represented by a fatgraph }
	\label{fig:fat04a}
\end{figure}

Thus the three branched covers arise from a unique branched cover which has ramification $\{(4,4),(2,2,2,2),(2,2,2,2)\}$ respectively over $\{0,1,\infty\}$.  This degree 8 example is Galois with Galois group a normal subgroup of order  8 in $S^4$ which acts on the labeled points.  
The covering map is invariant under the subgroup $\langle z\mapsto iz,1/z\rangle$ and is given by
\[ p(z)=\frac{-4z^4}{(z^4-1)^2}. \]
The point in $\modm_{0,4}$ is represented by $p^{-1}(\infty)$ = fourth roots of unity.  Different labelings of the points $p^{-1}(\infty)$ give rise to 3 different labeled branched covers (not 24 due to the symmetry.)  The cross-ratios of the points $p^{-1}(\infty)$ taken in various orders are $\{2,1/2,-1\}$ so these three algebraic numbers correspond to the curves counted by $N_{0,4}(2,2,2,2)=3$.  \\
\\
\fbox{$b=3$}  

$N_{0,4}(3,3,3,3)=8$ determines the fatgraph in figure~\ref{fig:fat04b} with edge lengths 1, and the fatgraph in figure~\ref{fig:fat04a} with edge lengths alternating between 1 and 2.  In both cases all boundary lengths are 3.  All labelings are used which leads to two, respectively six, labeled fatgraphs corresponding to figures~\ref{fig:fat04b} and \ref{fig:fat04a}.
\begin{figure}[ht]  
	\centerline{\includegraphics[height=2.5cm]{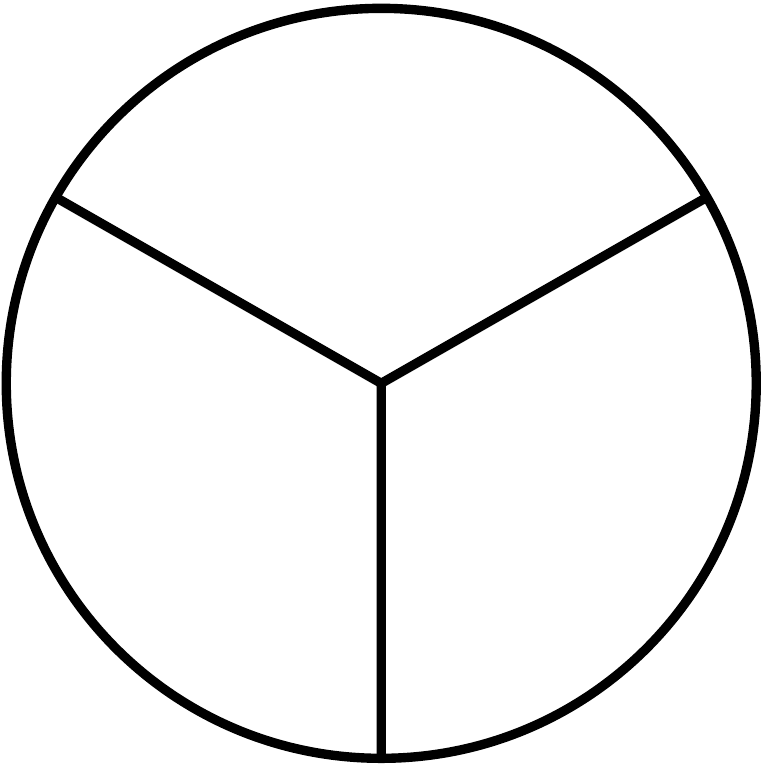}}
	\caption{Branched cover represented by a fatgraph }
	\label{fig:fat04b}
\end{figure}

\noindent The first branched cover has ramification $\{(3,3,3,3),(2,2,2,2,2,2),(3,3,3,3)\}$ respectively over $\{0,1,\infty\}$.  It is Galois with Galois group the full set of tetrahedral symmetries  and we use $K_{\pm}(z)$.  The covering map is
\[ p_1(z)=\frac{(z^4+2\sqrt{3}z^2+1)^3}{(z^4-2\sqrt{3}z^2+1)^3}. \]
The second branched cover has ramification $\{(2,2,4,4),(2,2,2,2,2,2),(3,3,3,3)\}$ respectively over $\{0,1,\infty\}$.  It has double cover with ramification $\{(4^6),(2^{12}),(3^8)\}$ which is Galois with Galois group the full set of octahedral symmetries.  The covering map $p_2$ satisfies $p_2(z^2)=108K_e^4/K_f^3$ and
\[ p_2(z)=\frac{108z^2(z^2-1)^4}{(z^4+14z^2+1)^3}. \]
The two points in $\modm_{0,4}$ represented by $p_1^{-1}(\infty)$ (using different labelings) have cross-ratios $\exp(\pm \pi i/3)$.   The six points in $\modm_{0,4}$ represented by $p_2^{-1}(\infty)$ (using different labelings) have cross-ratios $\{4,-3,1/4,-1/3,4/3,3/4\}$.  Hence the eight algebraic numbers correspond to the curves counted by $N_{0,4}(3,3,3,3)=8$ are 
\[\{4,-3,1/4,-1/3,4/3,3/4,\exp(\pi i/3),\exp(-\pi i/3)\}.\]  They are drawn as lattice points in cells of $\modm_{0,4}$ by gluing together two copies of Figure~\ref{fig:cell}.

The cross-ratios of the points $p^{-1}(\infty)$ and $p_1^{-1}(\infty)$ can be calculated without calculating $p$ and $p_1$.   The symmetric group $S_4$ acts on $\modm_{0,4}$ with generic orbit consisting of 6 points 
\[\left\{c,1-c,\frac{1}{c},\frac{1}{1-c},\frac{c}{c-1},\frac{c-1}{c}\right\}.\]  
Equivalently, a cross-ratio of 4 points can take on 6 different values depending on the order of the points.  There are two exceptional orbits of $S_4$ consisting of 3 points and 2 points, given by $\{-1,2,1/2\}$, respectively $\{\exp(\pi i/3),\exp(-\pi i/3)\}$, agreeing with the calculations above.
\end{ex}

\section{Laplace transforms}  \label{sec:lap}

The Kontsevich volume polynomial $V_{g,n}(b_1,...,b_n)$ is the first example in a collection of polynomials that give information about the moduli space of curves  and a Hurwitz problem.  Kontsevich studied $V_{g,n}(b_1,...,b_n)$ via its Laplace transform \cite{KonInt} and a matrix model.  The quasi-polynomial $N_{g,n}(k_1,...,k_n)$ and the polynomial term $P_{g,n}(k_1,...,k_n)$ in the ELSV formula are other examples.  They are both related to Kontsevich's volume polynomial $V_{g,n}(b_1,...,b_n)$.    

A fourth example related to Kontsevich's volume polynomial $V_{g,n}(b_1,...,b_n)$ is Mirzakhani's calculation of the Weil-Petersson volume of moduli space.
Let  $\modm_{g,n}({\bf L})$ be the moduli space of connected oriented genus $g$ hyperbolic surfaces with $n$ labeled geodesic boundary components of non-negative real lengths ${\bf L}=(L_1,...,L_n)$.  It comes equipped with a symplectic form $\omega$ which gives rise to the Weil-Petersson volume 
\[V^{WP}_{g,n}({\bf L})=\int_{\modm_{g,n}({\bf L})}\frac{\omega^N}{N!}\]
where $2N=\dim\modm_{g,n}({\bf L})=6g-6+2n$.
\begin{thm}[Mirzakhani \cite{MirSim,MirWei}]  \label{th:mirz}
$V^{WP}_{g,n}({\bf L})$ are polynomials in ${\bf L}=(L_1,...,L_n)$ satisfying
\[V^{WP}_{g,n}(L_1,...,L_n)=2^{-\chi}V_{g,n}(L_1,...,L_n)+{\rm lower\ order\ terms}\]
for $\chi=2-2g-n$.
\end{thm}
Mirzakhani proved the relationship between $V^{WP}_{g,n}$ and Kontsevich's volume polynomial $V_{g,n}$ by identifying the coefficients in the polynomial $V^{WP}_{g,n}$ with intersection numbers over the compactified moduli space.  One can also prove this relationship by considering $V^{WP}_{g,n}(NL_1,...,NL_n)$ as $N\to\infty$.   Hyperbolic surfaces with large geodesic boundary lengths necessarily become very thin to keep their area constant and hence converge to fatgraphs \cite{DoInt,MonTri}.
\begin{table}[ht]  \label{tab:volwp}
\caption{Weil-Petersson volumes}
\begin{spacing}{1.4}  
\begin{tabular}{||l|c|c||} 
\hline\hline
{\bf g} &{\bf n}&$V^{WP}_{g,n}(L_1,...,L_n)$\\ \hline
0&3&1\\ \hline
1&1&$\frac{1}{48}\left(L^2+4\pi^2\right)$\\ \hline
0&4&$\frac{1}{2}\left(L_1^2+L_2^2+L_3^2+L_4^2+4\pi^2\right)$\\ \hline
1&2&$\frac{1}{192}\left(L_1^2+L_2^2+4\pi^2\right)\left(L_1^2+L_2^2+12\pi^2\right)$\\ \hline
2&1&$\frac{1}{2^{14}3^35}\left(L_1^2+4\pi^2\right)\left(L_1^2+12\pi^2\right)\left(5L_1^4+384\pi^2L_1^2+6960\pi^4\right)$\\
\hline\hline
\end{tabular} 
\end{spacing}
\end{table}

\subsection{Total derivative of the Laplace transform}  \label{sec:totder}
It was shown in Section~\ref{sec:convpvl} that the Laplace transform brings a transparent structure to the volume and number of lattice points of a convex polytope.  The Laplace transform $\lc\{P_{g,n}\}(z_1,...,z_n)$ of $P_{g,n}=$ $V_{g,n}(b_1,....,b_n)$, $N_{g,n}(b_1,....,b_n)$, $H_{g,n}(\mu_1,...,\mu_n)$ or $V_{g,n}^{WP}(L_1,....,L_n)$ extends to a meromorphic function on the whole complex plane in each variable $z_i$.  The discrete Laplace transform is used for the enumerative problems and the continuous Laplace transform for volumes.  
\begin{defn}
For $P_{g,n}=$ $V_{g,n}(b_1,....,b_n)$, $N_{g,n}(b_1,....,b_n)$, $H_{g,n}(\mu_1,...,\mu_n)$ or $V_{g,n}^{WP}(L_1,....,L_n)$ define
\[\omega^g_n=\frac{\partial}{\partial z_1}...\frac{\partial}{\partial z_n}\lc\{P_{g,n}\}dz_1...dz_n.\]
\end{defn}
\noindent When $P_{g,n}=V_{g,n}(b_1,....,b_n)$,
\begin{align*}
\omega^0_3[{\rm Airy}]&=-\frac{1}{2}\frac{dz_1dz_2dz_3}{z_1^2z_2^2z_3^2}\\
\omega^1_1[{\rm Airy}]&=-\frac{dz}{16z^4}\\
\omega^0_4[{\rm Airy}]&=\frac{1}{2}\frac{dz_1dz_2dz_3dz_4}{z_1^2z_2^2z_3^2z_4^2}\left(\frac{1}{z_1^2}+\frac{1}{z_2^2}+\frac{1}{z_3^2}+\frac{1}{z_4^2}\right)
\end{align*}
and the poles of $\omega^g_n[{\rm Airy}]$ occur at $z_i=0$.  We associate the term Airy with Kontsevich's volume polynomial in anticipation of the framework of Eynard and Orantin described later.
When $P_{g,n}=N_{g,n}(b_1,....,b_n)$,
\begin{align*}
\omega^0_3&=\left\{\frac{1}{2\prod(1-z_i)^2}-\frac{1}{2\prod(1+z_i)^2}\right\}\prod dz_i\\
\omega^1_1&=\frac{z^3dz}{(1-z^2)^4}\\
\omega^0_4&=\left\{\frac{3}{4\prod(1-z_i)^2}\sum\frac{z_i}{(1-z_i)^2}-\frac{3}{4\prod(1+z_i)^2}\sum\frac{z_i}{(1+z_i)^2}\right.\\
&\left.\quad+\frac{\sum z_iz_j(1+z_k^2)(1+z_l^2)}{2\prod(1-z_i^2)^2}\right\}\prod dz_i.\\
\end{align*}
The poles of $\omega^g_n$ occur at $z_i=\pm 1$.  The asymptotic behaviour of $\omega^g_n$ near its poles $z_i=\pm 1$ is $\omega^g_n[{\rm Airy}]$.  More precisely, consider the change of variables $z_i=1+sx_i$ where $x_i$ is a local coordinate in a neighbourhood of the pole.  The dominant asymptotic term as $s\rightarrow 0$ is
\begin{equation} \label{eq:asym}
\omega^g_n\sim s^{6-6g-3n}\omega^g_n[Airy].
\end{equation}
The same asymptotic behaviour occurs around $z_i=-1$.  This case is a consequence of the relationship between the continuous and discrete Laplace transforms.  What is striking is that Eynard and Orantin have developed a framework arising out of the study of matrix models \cite{EOrInv} that unifies these different problems and explains the appearance of Kontsevich's volume polynomial.

\subsection{Eynard-Orantin invariants}   \label{sec:eo}

A Hermitian {\em matrix model} is an integral over the space $H_N$ of $N\times N$ Hermitian matrices
\[ \int_{H_N}\exp\{V(M)\}d\mu(M)\]
for a function $V:H_N\to\bc$, known as a potential, with respect to a Gaussian measure $d\mu$.   Matrix models are closely related to graphical enumeration \cite{BIZQua} and to the cell decomposition of the moduli space of curves beginning with calculations of the Euler characteristic of the moduli space of curves by Harer and Zagier \cite{HZaEul} and Penner \cite{PenPer} and intersection numbers on the moduli space of curves by Kontsevich \cite{KonInt}.  Kontsevich proved that an asymptotic expansion of a matrix model gives a generating function for the Laplace transforms of the volume polynomials $V_{g,n}(b_1,...,b_n)$ which he related to intersection numbers on $\overline{\modm}_{g,n}$.  Chekhov \cite{CheMat,CheMat1} studied a matrix model related to $N_{g,n}(k_1,...,k_n)$ and produced the discrete Laplace transforms of $N_{0,3}(b_1,b_2,b_3)=1$ and $N_{1,1}(b)=\frac{1}{48}\left(b^2-4\right)$.

The expected location of eigenvalues of a matrix model for large matrices gives rise to a Riemann surface, known as the spectral curve associated to the matrix model \cite{EOrTop}.  Eynard and Orantin \cite{EOrInv} use the ideas from matrix models to define invariants of plane curves without referring to a matrix model.  They associate multidifferentials $\omega^g_n$ to any Torelli marked Riemann surface $C$ equipped with two meromorphic functions $x$ and $y$ with the property that the branch points of $x$ are simple and the map
\[ \begin{array}[b]{rcl} C&\to&\bc^2\\p&\mapsto& (x(p),y(p))\end{array}\]
is an immersion.  Recall that a {\em Torelli marking} of $C$ is a choice of symplectic basis $\{a_i, b_i\}_{i=1,..,g}$ of the first homology group $H_1(\bar{C})$ of the compact closure $\bar{C}$ of $C$.

For every $(g,n)\in\bz^2$ with $g\geq 0$ and $n>0$ Eynard and Orantin \cite{EOrInv} define a multidifferential, i.e. a tensor product of meromorphic 1-forms on the product $C^n$, denoted by $\omega^g_n(p_1,...,p_n)$ for $p_i\in C$.  When $2g-2+n>0$, $\omega^g_n(p_1,...,p_n)$ is defined recursively in terms of local  information around the poles of $\omega^{g'}_{n'}(p_1,...,p_n)$ for $2g'+2-n'<2g-2+n$.  

For $2g-2+n>0$, the poles of $\omega^g_n(p_1,...,p_n)$ occur at the zeros of $dx$.  Since each branch point $\alpha$ of $x$ is simple, for any point $p\in C$ close to $\alpha$ there is a unique point $\hat{p}\neq p$ close to $\alpha$ such that $x(\hat{p})=x(p)$.  The recursive definition of $\omega^g_n(p_1,...,p_n)$ uses only local information around zeros of $dx$ and makes use of the well-defined map $p\mapsto\hat{p}$ there.  Equivalently, the $\omega^{g'}_{n'}(p_1,...,p_n)$ are used as kernels on the Riemann surface appearing as part of an integrand.  This is a familiar idea, the main example being the Cauchy kernel which gives the derivative of a rational function in terms of the bidifferential $dwdz/(w-z)^2$ as follows
\[ f'(z)dz=\res{w=z}\frac{f(w)dwdz}{(w-z)^2}=-\sum_{\alpha}\res{w=\alpha}\frac{f(w)dwdz}{(w-z)^2}\]
where the sum is over all poles $\alpha$ of $f(w)$.  

The Cauchy kernel generalises to a bidifferential $B(w,z)$ on any Torelli marked Riemann surface $C$ given by the meromorphic differential $\eta_w(z)dz$ unique up to scale which has a double pole at $w\in C$ and all $A$-periods vanishing.   The scale factor can be chosen so that $\eta_w(z)dz$ varies holomorphically in $w$, and transforms as a 1-form in $w$ and hence it is naturally expressed as a unique bidifferential $B(w,z)$ on $C$.  More precisely, define the {\em Bergmann kernel} $B(w,z)$ by

(i) $B(w,z)=B(z,w)$,

(ii) $B(w,z)$ has a second order pole on the diagonal $w=z$ and biresidue 1,

(iii) the $A$-periods of $B(w,z)$ vanish.

The term {\em Bergmann kernel} was introduced by Tyurin \cite{TyuPer}.  It is called the fundamental normalised differential of the second kind on $C$ in \cite{FayThe}.  Recall that a meromorphic differential is {\em normalised} if its $A$-periods vanish and it is of the {\em second kind} if its residues vanish.  It is used to express a normalised differential of the second kind in terms of local  information around its poles.  Define
\begin{equation} \label{eq:berg}
\omega^0_1=-ydz,\quad \omega^0_2=B(w,z)
\end{equation}
and for $2g-2+n>0$,
\begin{equation}  \label{eq:EOrec}
\omega^g_{n+1}(z_0,z_S)=\sum_{\alpha}\hspace{0mm}\res{z=\alpha}\ K(z_0,z)\hspace{0mm}\biggr[\omega^{g-1}_{n+2}(z,\hat{z},z_S)+\hspace{-5mm}\displaystyle\sum_{\begin{array}{c}_{g_1+g_2=g}\\_{I\sqcup J=S}\end{array}}\hspace{-2mm}
\omega^{g_1}_{|I|+1}(z,z_I)\omega^{g_2}_{|J|+1}(\hat{z},z_J)\biggr]
\end{equation}
where the sum is over branch points $\alpha$ of $x$, $S=\{1,...,n\}$, $I$ and $J$ are non-empty and 
\[\displaystyle K(z_0,z)=\frac{-\int^z_{\hat{z}}B(z_0,z')}{2(y(z)-y(\hat{z}))dx(z)}\] is well-defined in the vicinity of each branch point of $x$.   Note that the quotient of a differential by the differential $dx(z)$ is a meromorphic function.

Eynard and Orantin proved \cite{EOrTop} that the invariant $\omega^g_n$ for any plane curve $C$ behaves asymptotically like (\ref{eq:asym}) near its poles.  The proof of this asymptotic behaviour starts with the simple observation that near each branch point the plane curve $C$ locally resembles the Airy curve $y^2=x$.   In particular the Kontsevich volume $V_{g,n}(b_1,...,b_n)$ is intimately related to the invariant for any curve. 

The total derivatives of the Laplace transform of the Kontsevich volume polynomial $V_{g,n}(b_1,...,b_n)$ define the Eynard-Orantin invariants $\omega^g_n$ for the curve $x=z^2$, $y=z$ known as the Airy curve, \cite{EOrInv}.   The total derivatives of the discrete Laplace transform of the lattice count quasi-polynomial $N_{g,n}(b_1,...,b_n)$ define the Eynard-Orantin invariants $\omega^g_n$ for the curve $x=z+1/z$, $y=z$,  \cite{NorStr}.  The total derivatives of the discrete Laplace transform of the simple labeled Hurwitz numbers $H_{g,n}(\mu_1,...,\mu_n)$ define the Eynard-Orantin invariants $\omega^g_n$ for the Lambert curve $x=-z+\ln(z/t)$, $y(z)=z$, \cite{BEMSMat,BMaHur,EMSLap}.  The total derivatives of the Laplace transform of the Weil-Petersson volume polynomial $V_{g,n}^{WP}(b_1,...,b_n)$ define the Eynard-Orantin invariants $\omega^g_n$ for the curve $x=z^2$, $y=-\sin(2\pi z)/4\pi$, \cite{EOrWei}.  The last two cases represent sequences of algebraic curves obtained by truncating the expansions for $y$ around the zeros of $dx$.  In each case the asymptotic behaviour of $\omega^g_n$ near its poles explains the relationship with $V_{g,n}(b_1,....,b_n)$.  A large class of polynomials with highest degree terms given by $V_{g,n}(b_1,....,b_n)$ coming from Eynrad-Orantin invariants is studied in \cite{NScPol}.

Various other Hurwitz problems appear in the literature.  When $\Sigma'=T^2$ and all branch points are simple the generating function for the Hurwitz numbers over all covering degrees is a modular form \cite{KonHom}.  In \cite{OPaGro}, simple branching is generalised to the Hurwitz problem of covers of $S^2$ unramified on $S^2-\{p_1,...,p_r\}$ with ramification $(b_i,1^{d_i})$ over $p_i$ and this Hurwitz problem is related to the Gromov-Witten invariants of $\bp^1$.   The Gromov-Witten invariants of $\bp^1$ assemble to give the Eynard-Orantin invariants $\omega^g_n$ for the curve $x=z+1/z$, $y=\ln{z}$, \cite{NScGro}.  Double Hurwitz numbers are defined to be the weighted count of genus $g$ connected covers of $S^2$ with ramification $\mu$ over $\infty$, $\nu$ over $0$ and simple ramification elsewhere \cite{GJVTow}. Given a partition $\mu$ of $2d$ and a partition $\nu$ of $2d$ into odd parts, \cite{EOkPil} considers the Hurwitz problem of branched covers of $S^2$ with ramification over $0$ given by $\left(\nu,2^{d-|\nu|/2}\right)$, ramification $\left(2^d\right)$ over 3 other points, and ramification $\left(\mu_i,1^{2d-\mu_i}\right)$ over $l(\mu)$ given points.  Three of the Hurwitz problems described here---simple Hurwitz numbers, Gromov-Witten invariants of $\bp^1$ and the Belyi Hurwitz problem---have a common formulation in terms of Eynard-Orantin invariants.  It would be interesting to find such formulations for other Hurwitz problems.  This could bring further understanding to the Eynard-Orantin invariants and uncover relations between Hurwitz problems and the moduli space of Riemann surfaces.

\end{document}